\documentclass{amsart}

\usepackage{tikz}
\usetikzlibrary{arrows.meta}
\usepackage{commutative-diagrams}
\usetikzlibrary{cd}
\usepackage{geometry}
\usepackage[colorlinks=true]{hyperref}
\usepackage[all]{hypcap}
\usepackage{mathtools}
\usepackage{booktabs}
\usepackage{tabularx}
\usepackage{array}
\usepackage{amssymb}
\usepackage{dsfont}
\usepackage{witharrows}
\usepackage{threeparttable}
\usepackage{enumitem}

\newtheorem{theorem}{Theorem}[section]
\newtheorem{lemma}[theorem]{Lemma}
\newtheorem{corollary}[theorem]{Corollary}

\theoremstyle{definition}
\newtheorem{definition}[theorem]{Definition}
\newtheorem{example}[theorem]{Example}

\theoremstyle{remark}
\newtheorem{remark}[theorem]{Remark}

\numberwithin{equation}{section}

\newcolumntype{C}{>{\centering\arraybackslash}X}

\begin{document}

\title[Orbit span of a design and some saturation theorems]{Orbit span of a design and some saturation theorems in commutative Schurian association schemes}


\author[Q. Li]{Qilong Li}
\address{College of Science, National University of Defense Technology, 410073 Changsha, China}
\email[Q. Li]{li.qilong@outlook.com}

\author[Y. Zhou]{Yue Zhou}
\address{College of Science, National University of Defense Technology, 410073 Changsha, China}
\email[Y. Zhou]{yue.zhou.ovgu@gmail.com}



\begin{abstract}
	Orbit-span and dimension problems for designs have been studied in several classical association schemes using a variety of methods. Recently, through a detailed analysis of total trades, Ghorbani et al. showed that the orbit of a fixed combinatorial design asymptotically attains the full dimension permitted by the design equations. Analogous dimension-saturation results for the global spans of index-one designs in the bilinear forms and Grassmann schemes were obtained via laborious eigenvalue computations.\par 
	In this paper, we work with the top fiber $X$ of a graded poset carrying a compatible transitive action of a finite group $G$, and assume that the induced Schurian association scheme on $X$ is commutative. For the multiplicity-free decomposition $\mathbb{C}^{X}=U_{0}\oplus U_{1}\oplus\cdots\oplus U_{m}$, we prove that, under explicit spectral and quantitative conditions, the $G$-orbit of the characteristic vector of any $t$-design spans the maximal submodule $U_{0}\oplus U_{t+1}\oplus U_{t+2}\oplus\cdots\oplus U_{m}$ allowed by the $t$-design definition. We verify these conditions asymptotically for the Hamming, bilinear forms, Johnson, and Grassmann schemes. This recovers the fixed-orbit saturation theorem of Ghorbani et al. for combinatorial designs, gives a new result for orthogonal arrays, and asymptotically extends the previous bilinear forms and Grassmann results to orbit spans of individual designs of arbitrary fixed index.
\end{abstract}

\maketitle

\section{Introduction}
Let $t,v,n$ and $\lambda$ be positive integers satisfying $t\leq n\leq v/2$. A \textit{combinatorial $t$-$(v,n,\lambda)$ design} $D$ (or simply a $t$-design) is a collection of $n$-subsets of $[v]:=\left\{1,2,\ldots,v\right\}$, where each element of $D$ is called a \textit{block}, such that every $t$-subset of $[v]$ is contained in exactly $\lambda$ blocks in $D$. In the sequel, we assume that $D$ has no repeated blocks unless otherwise specified.\par 
The concept of combinatorial designs can be described in a linear-algebraic language using the \textit{inclusion matrix} $W_{t,n}^{v}$, which is a $(0,1)$-matrix whose rows and columns are indexed by $t$-subsets and $n$-subsets of $[v]$, respectively, and an entry $W_{t,n}^{v}(T,K)=1$ if and only if $T\subseteq K$. For a collection $D$ of $n$-subsets of $[v]$, we define its \textit{characteristic vector} $\boldsymbol{1}_{D}$ to be the $(0,1)$-vector indexed by all $n$-subsets of $[v]$, where an entry $\boldsymbol{1}_{D}(B)=1$ if and only if $B\in D$. Therefore, $D$ is a $t$-design if and only if $\boldsymbol{1}_{D}$ is a $(0,1)$-solution for the linear system
\begin{align}\label{eq_linear.system.of.inclusion.matrix.defining.designs}
	W_{t,n}^{v}x=\lambda\boldsymbol{1},
\end{align}
where $\boldsymbol{1}$ stands for the all-ones vector. If we view \eqref{eq_linear.system.of.inclusion.matrix.defining.designs} as a linear system over $\mathbb{Q}$, then every integral solution for \eqref{eq_linear.system.of.inclusion.matrix.defining.designs} is called a \textit{signed $t$-design of index $\lambda$}. Throughout we identify a subset with its characteristic vector.\par 
A necessary divisibility condition for the existence of a combinatorial $t$-$(v,n,\lambda)$ design is that ${n-i\choose t-i}\mid\lambda{v-i\choose t-i}$ for all $i=0,1,..,t$. For signed designs, these divisibility conditions are also sufficient \cite{graver1973module,wilson1973necessary}. A classical result of Gottlieb \cite{gottlieb1966certain} shows that $W_{t,n}^{v}$ has full rank over $\mathbb{Q}$. Consequently, whenever the system \eqref{eq_linear.system.of.inclusion.matrix.defining.designs} is consistent, the linear space generated by all rational solutions of \eqref{eq_linear.system.of.inclusion.matrix.defining.designs} over $\mathbb{Q}$ has dimension ${v\choose n}-{v\choose t}+1$. A natural question asks whether the $(0,1)$-solutions already generate this entire space. Ghodrati \cite{ghodrati2019dimension} answered this affirmatively in 2019 for $\lambda=1$ by showing that the characteristic vectors of all $t$-$(v,n,1)$ designs span the whole rational solution space, and hence also the space of signed $t$-designs of index one. More recently, Ghorbani et al. \cite{ghorbani2025vector} proved a substantially stronger fixed-design theorem, stating that for any fixed $t$-$(v,n,\lambda)$ design which is not a $(t+1)$-design, the span of its permutations has the maximum possible dimension ${v\choose n}-{v\choose t}+1$ once $v$ is sufficiently large. By Keevash's existence theorem \cite{keevash2014existence}, for fixed $t,n,\lambda$ the usual divisibility conditions are asymptotically sufficient, so these spanning statements occur throughout the admissible large-$v$ regime.\par 
Although the final dimension formula in the Johnson scheme is simple, the existing proofs are rather technically and specialized. Ghorbani et al. \cite{ghorbani2025vector} worked with a direct-sum decomposition of the orbit span of a fixed $t$-$(v,n,\lambda)$ design into subspaces generated by \textit{total trades}
\begin{align*}
	\operatorname{span}_{\mathbb{Q}}\left\{\sigma\boldsymbol{1}_{\mathbf{S}}\mid\sigma\in\mathfrak{S}_{v}\right\}=\langle\boldsymbol{1}\rangle\oplus\bigoplus_{i\in I}\langle\mathfrak{T}_{i,n,v}\rangle
\end{align*}
for some index set $I\subseteq\left\{t,\ldots,n-1\right\}$. Here $\mathfrak{S}_{v}$ denotes the symmetric group on $[v]$ and each $\langle\mathfrak{T}_{i,n,v}\rangle$ is a subspace spanned by the so-called \textit{total trades}. Each subspace $\langle\mathfrak{T}_{i,n,v}\rangle$ ($i<n$) has dimension ${v\choose i+1}-{v\choose i}$. These total trade subspaces turn out to be irreducible modules of $\mathfrak{S}_{v}$ and are isomorphic to the \textit{Specht modules} \cite{maliakas2026total}. This method is closely tied to the subset lattice and the symmetric-group action.\par 
From the viewpoint of association schemes, it is natural to ask whether the spanning phenomena extend to designs in more general association schemes. Indeed, if we move from the \textit{Johnson scheme} underlying combinatorial designs to its $q$-analog, namely the \textit{$q$-Johnson scheme} (also known as the \textit{Grassmann scheme}), the authors and C. Wei{\ss} \cite{li2026dimension.q-steiner} showed that the 2019 work \cite{ghodrati2019dimension} of Ghodrati has its precise analog for subspace designs with $\lambda=1$. An analogous result for the bilinear forms schemes was obtained by the authors in \cite{li2026dimension.mrd}; the proofs of these results are based on laborious computations of specific eigenvalues and do not provide a direct geometric intuition for the structure of the spanning space.\par 
In this paper, we study the linear space spanned by the orbit of any fixed design $D$ under a transitive group $G$ inducing the association scheme $\left(X,R\right)$, where $\left(X,R\right)$ arises from the group action of $G$ (which is called a \textit{Schurian association scheme}) and is commutative. We also assume a partial order defined on $X$ which enjoys some regularity and compatibility conditions. More specifically, if we write the orthogonal decomposition $\mathbb{C}^{X}=U_{0}\oplus U_{1}\oplus\cdots\oplus U_{m}$ of the permutation module $\mathbb{C}^{X}$, then we deduce the decomposition
\begin{align*}
	\operatorname{span}_{\mathbb{C}}\left\{g\boldsymbol{1}_{D}\mid g\in G\right\}=U_{0}\oplus U_{t+1}\oplus U_{t+2}\oplus\cdots\oplus U_{m}
\end{align*}
in Theorem \ref{thm_general.saturation.theorem}. In other words, under certain assumptions we show that the orbit span of any fixed $t$-design $D$ saturates all the irreducible components $U_{j}$ except the first $t$ non-trivial ones.\par 
In particular, we apply this general theorem to the Hamming, bilinear forms, Johnson, and Grassmann schemes in Sections \ref{sect_hamming.and.bilinear.forms} and \ref{sect_johnson.and.grassmann}, by verifying the hypotheses in our general saturation theorem for these four classical association schemes. As a consequence, in each case we determine the dimension of the orbit span of every fixed design once the relevant parameter is sufficiently large. The Johnson-scheme result recovers the theorem of Ghorbani et al. \cite{ghorbani2025vector}, while the bilinear-forms and Grassmann results strengthen the corresponding global spanning results in \cite{li2026dimension.mrd,li2026dimension.q-steiner}.

\section{Preliminaries}
\subsection{Commutative Schurian association schemes}
Let $X$ be a finite set and let $G$ be a finite group that acts transitively on $X$. Let $R=\left\{R_{0},R_{1},\ldots,R_{m}\right\}$ be the set of the orbits of $G$ on $X\times X$ under component-wise action, where $R_{0}=\left\{(x,x)\mid x\in X\right\}$. Then $\left(X,R\right)$ forms an association scheme and it is called a \textit{Schurian association scheme}. If $\mathbb{C}^{X}$ is multiplicity-free as a $\mathbb{C}[G]$-module, i.e. there exists an orthogonal decomposition
\begin{align}\label{eq_multiplicity.free.association.scheme.decomposition}
	\mathbb{C}^{X}=U_{0}\oplus U_{1}\oplus\cdots\oplus U_{m}
\end{align}
of $\mathbb{C}^{X}$ into pairwise non-isomorphic irreducible $G$-modules $U_{i}$, where $U_{0}=\mathbb{C}\boldsymbol{1}$, then we say that the Schurian association scheme $\left(X,R\right)$ is \textit{multiplicity-free}. This terminology follows the standard convention in representation theory. We refer to the monograph \cite{gordon2001representations} for the background. We mention the following result (see \cite[Chapter 9]{gordon2001representations}), which is fundamental in representation theory and will be used frequently in the sequel.
\begin{lemma}[Schur's lemma]
	Let $V$ and $W$ be irreducible $\mathbb{C}[G]$-modules.
	\begin{enumerate}[itemsep=5pt]
		\item[(i)] If $\vartheta:V\to W$ is a $\mathbb{C}[G]$-homomorphism, then either $\vartheta$ is a $\mathbb{C}[G]$-isomorphism, or $v\vartheta=0$ for all $v\in V$;
		\item[(ii)] If $\vartheta:V\to V$ is a $\mathbb{C}[G]$-isomorphism, then $\vartheta$ is a scalar multiple of the identity endomorphism on $V$.
	\end{enumerate}
\end{lemma}
Let $x_{0}$ be any fixed element in $X$ and let $K:=G_{x_{0}}=\left\{g\in G\mid gx_{0}=x_{0}\right\}$ be its stabilizer. Then the set $X$ can be identified with the quotient $G/K$. Let $\mathds{1}_{K}$ be the trivial character on $K$. Then the induced representation of $\mathds{1}_{K}$ is
\begin{align*}
	\operatorname{Ind}_{K}^{G}\mathds{1}_{K}=\left\{f\in\mathbb{C}[G]\mid f(gk)=f(g)~\text{for all $g\in G$, $k\in K$}\right\},
\end{align*}
which is the set of the complex-valued functions on $G$ that are constant on cosets of $K$, i.e. $\mathbb{C}[G/K]$. Therefore, the induced representation $\operatorname{Ind}_{K}^{G}\mathds{1}_{K}$ can be identified with $\mathbb{C}^{X}$. By Schur's lemma, $\operatorname{Ind}_{K}^{G}\mathds{1}_{K}$ is multiplicity-free if and only if $\operatorname{End}_{G}\left(\operatorname{Ind}_{K}^{G}\mathds{1}_{K}\right)$ is commutative. Therefore, by noticing that $\operatorname{End}_{G}(\mathbb{C}^{X})$ is exactly the Bose--Mesner algebra of $(X,R)$ since $\mathbb{C}^{X}$ is the permutation module of $G$ on $X$, we conclude that the Schurian association scheme $(X,R)$ is multiplicity-free if and only if it is commutative. The pair $(G,K)$ here is then called a \textit{Gelfand pair}. We refer readers to \cite[Chapter VII]{macdonald1995symmetric}, \cite{ceccherini2009finite} and \cite[\S 3]{martin2009commutative} for more detailed treatments.
\begin{example}[{\cite[\S 2.2]{ceccherini2009finite}}]\label{example_schurian.commutative.association.scheme.from.gelfand.pair.and.two.point.homogeneous.space}
	Let $G$ be a finite group that acts isometrically on a finite metric space $\left(X,d\right)$ and suppose that the action is distance transitive, i.e., for all $x_{1},x_{2},y_{1},y_{2}\in X$ such that $d(x_{1},y_{1})=d(x_{2},y_{2})$ there exists $g\in G$ such that $gx_{1}=x_{2}$ and $gy_{1}=y_{2}$. Then for any fixed $x_{0}\in X$ and its stabilizer $K=G_{x_{0}}$, the pair $(G,K)$ is a \textit{symmetric Gelfand pair}, which means the Schurian association scheme $(X,R)$ is both symmetric and commutative.
\end{example}
The following well-known examples of association schemes are special cases of Example \ref{example_schurian.commutative.association.scheme.from.gelfand.pair.and.two.point.homogeneous.space} and are all Schurian, commutative and symmetric.
\begin{example}
	The Hamming scheme $H(n,q)$ is obtained from the Gelfand pair $(\mathfrak{S}_{q}\wr\mathfrak{S}_{n},\mathfrak{S}_{q-1}\wr\mathfrak{S}_{n})$.
\end{example}
\begin{example}
	The bilinear forms scheme $\operatorname{Bil}_{q}(m,m')$ is obtained from the Gelfand pair $(\mathbb{F}_{q}^{m'\times m}\rtimes(\operatorname{GL}(m,q)\times\operatorname{GL}(m',q)),\operatorname{GL}(m,q)\times\operatorname{GL}(m',q))$.
\end{example}
\begin{example}
	The Johnson scheme $J(v,n)$ is obtained from the Gelfand pair $(\mathfrak{S}_{v},\mathfrak{S}_{n}\times\mathfrak{S}_{v-n})$.
\end{example}
\begin{example}
	The Grassmann scheme $J_{q}(n,k)$ is obtained from the Gelfand pair $(\operatorname{GL}(n,q),P_{k})$, where $P_{k}=\left\{\begin{pmatrix}
		A & B\\
		0 & D
	\end{pmatrix}\mid A\in\operatorname{GL}(k,q),D\in\operatorname{GL}(n-k,q),B\in\mathbb{F}_{q}^{k\times(n-k)}\right\}$ is the parabolic subgroup.
\end{example}
\subsection{Graded partially ordered sets and the incidence operators}
A \textit{partially ordered set} (or a \textit{poset}) $P$ is a set together with a binary relation denoted $\leq$ satisfying the reflexivity, antisymmetry and transitivity axioms (see \cite[\S 3.1]{stanley2012enumerative}). The poset $P$ is said to be \textit{graded} if it is equipped with a \textit{rank function} $\operatorname{rk}:P\to\mathbb{N}$ which satisfies the following conditions:
\begin{enumerate}[itemsep=5pt]
	\item If $x\leq y$, then $\operatorname{rk}(x)\leq\operatorname{rk}(y)$;
	\item If $x<y$ and no element $z\in P$ satisfies $x<z<y$, then $\operatorname{rk}(y)=\operatorname{rk}(x)+1$.
\end{enumerate}
For every graded poset $P$ with rank function $\operatorname{rk}$ we have a partition $P=X_{0}\sqcup X_{1}\sqcup\cdots\sqcup X_{m}$, where $m=\max_{x\in P}\operatorname{rk}(x)$ is said to be the \textit{rank} of $\left(P,\leq\right)$. We call each $X_{i}$ the \textit{$i$-th fiber} of $P$ and $X:=X_{m}$ is called the \textit{top fiber}.\par 
For $0\leq i\leq m$, we define the incidence operator $W_{i}:\mathbb{C}^{X}\to\mathbb{C}^{X_{i}}$ by setting
\begin{align*}
	\left(W_{i}f\right)(z):=\sum_{x\in X,x\geq z}f(x)
\end{align*}
for all $z\in X_{i}$. This is a variant of the Radon transform on a finite poset; see \cite{kung1993radon}. Let $W_{i}^*$ be the \textit{adjoint operator} of $W_{i}$. This means, for each $h\in\mathbb{C}^{X_{i}}$ we have
\begin{align*}
	\left(W_{i}^*h\right)(x)=\sum_{z\in X_{i},z\leq x}h(z),\quad x\in X.
\end{align*}
Throughout, we use the standard Hermitian inner product $\langle f,g\rangle=\sum_{x\in X}f(x)\overline{g(x)}$ and write $\|f\|:=\sqrt{\langle f,f\rangle}$ unless otherwise specified.
\subsection{Some estimates on the (Gaussian) binomial coefficients}
The following standard fact about the binomial coefficients will be used. This can be derived from Stirling's formula.
\begin{lemma}
	If $k$ is a constant, then ${n\choose k}=\left(1+o_{n}(1)\right)n^{k}/k!$.
\end{lemma}
Let $q$ be a prime power. For any integer $n$ and nonnegative integer $k$,
\begin{align*}
	{n\brack k}_{q}:=\frac{(q^{n}-1)(q^{n-1}-1)\cdots(q^{n-k+1}-1)}{(q^{k}-1)(q^{k-1}-1)\cdots(q-1)}
\end{align*}
is called the \textit{Gaussian binomial coefficient} (or $q$-\textit{binomial coefficient}). We define ${n\brack 0}_{q}=1$ whereas ${n\brack k}_{q}=0$ for all $k<0$ and $k>n$. The following inequalities on ${n\brack k}_{q}$ will also be used.
\begin{lemma}\label{q.binomial.estimate}
	For any positive integers $k\leq n$ we have
	\begin{align*}
		q^{k(n-k)}\leq{n\brack k}_{q}<cq^{k(n-k)},
	\end{align*}
	where $c=\prod_{i=1}^{\infty}\left(1-2^{-i}\right)^{-1}\approx 3.463$. In particular, we have ${n\brack k}_{q}=\Theta\left(q^{k(n-k)}\right)$.
	\begin{proof}
		The first inequality is obvious, since $(q^{n-i}-1)/(q^{k-i}-1)\geq q^{n-k}$ holds for all $0\leq i\leq k-1$. For the second inequality, we note that
		\begin{equation}
			\begin{split}
				{n\brack k}_{q}&=\frac{(q^{n}-1)(q^{n-1}-1)\cdots(q^{n-k+1}-1)}{(q^{k}-1)(q^{k-1}-1)\cdots(q-1)}\\
				&\leq\frac{q^{kn-{k\choose 2}}}{q^{{k+1\choose 2}}\prod_{i=0}^{k-1}(1-\frac{1}{q^{k-i}})}\\
				&\leq\frac{q^{k(n-k)}}{\prod_{i=0}^{k-1}(1-\frac{1}{2^{k-i}})}\\
				&<cq^{k(n-k)},\nonumber
			\end{split}
		\end{equation}
		where the constant $c$ is sometimes referred to as a \textit{digital search tree constant}. For the value of this number we refer readers to \cite[\S 5.14]{finch2003mathematical}.
	\end{proof}
\end{lemma}

\section{A saturation theorem for a fixed design in the association schemes}
For a commutative Schurian association scheme with underlying set $X$, we prove the following result on the orbit span of every vector in the permutation module $\mathbb{C}^{X}$.
\begin{lemma}[Orbit-span lemma]\label{lemma_orbit.span.lemma}
	With notation as above, for every $a\in\mathbb{C}^{X}$, we have
	\begin{align*}
		\operatorname{span}_{\mathbb{C}}\left(Ga\right)=\bigoplus_{0\leq i\leq m\atop\operatorname{proj}_{U_{i}}a\neq 0}U_{i}.
	\end{align*}
	\begin{proof}
		Set $A=\operatorname{span}_{\mathbb{C}}\left\{ga\mid g\in G\right\}$. Then $A$ is a $\mathbb{C}[G]$-submodule of $\mathbb{C}^{X}$. By Maschke's theorem, $A$ is semisimple. By \eqref{eq_multiplicity.free.association.scheme.decomposition} and the multiplicity-freeness of $\mathbb{C}^{X}$, every irreducible constituent of $A$ equals one of the $U_{i}$, and hence
		\begin{equation*}
			A=\bigoplus_{i\in I}U_{i}
		\end{equation*}
		for some $I\subseteq\left\{0,1,\ldots,m\right\}$. Since $a\in A$, its orthogonal decomposition contains only constituents occurring in $A$, therefore $\operatorname{proj}_{U_{i}}a=0$ for all $i\notin I$. Conversely, if $i\in I$, then we claim that $\operatorname{proj}_{U_{i}}a\neq 0$. Suppose for contradiction that $\operatorname{proj}_{U_{i}}a=0$. Since $\operatorname{proj}_{U_{i}}$ is $G$-equivariant, for every $g\in G$ we have $\operatorname{proj}_{U_{i}}(ga)=g\operatorname{proj}_{U_{i}}(a)=0$. Thus $\operatorname{proj}_{U_{i}}$ vanishes on every vector $ga$, and therefore it vanishes on $A$. On the other hand, $\operatorname{proj}_{U_{i}}$ is the identity on $U_{i}\subseteq A$, which is a contradiction. This implies $I=\left\{0\leq i\leq m\mid\operatorname{proj}_{U_{i}}a\neq 0\right\}$ and the proof is complete.
	\end{proof}
\end{lemma}
To establish our result on the orbit span of an arbitrary design in $X$, we need to assume the following graded poset structure of $X$. We note that these assumptions on the partial order $\leq$ are weaker than those for the \textit{regular semilattices} introduced by Delsarte \cite{delsarte1976association}, except here we require the group $G$ to satisfy some compatibility conditions.
\begin{definition}\label{def_G.admissible.graded.poset}
	Let $\left(\Omega,\leq\right)$ be a graded finite poset and let $G$ be a finite group that acts on $\Omega$. Let $X_{i}$ be its $i$-th fiber and let $m$ be the rank of $\Omega$. Then $\Omega$ is a \textit{$G$-admissible graded poset} if it enjoys the following properties:
	\begin{enumerate}[itemsep=5pt]
		\item For all $0\leq i\leq m$, the cardinality $\#\left\{z\in X_{i}\mid z\leq x\right\}$ is independent of the choice of $x\in X_{m}$ and we denote this number by $\alpha_{i}$;
		\item For all $0\leq t<i\leq m$ and every $z\in X_{i}$, there exists at least one $y\in X_{t}$ with $y\leq z$;
		\item The partial order is preserved under the action of $G$, i.e. for all $z,x\in\Omega$ and $g\in G$, it holds that $z\leq x$ if and only if $gz\leq gx$;
		\item The group $G$ acts transitively on the top fiber $X_{m}$ of $\Omega$.
	\end{enumerate}
\end{definition}
Consider the top fiber $X:=X_{m}$. We define a \textit{$t$-design} in $X$ as follows.
\begin{definition}
	We call a non-empty subset $D\subseteq X$ a (simple) $t$-design of index $\lambda$, where $t\in\left\{0,1,\ldots,m\right\}$, if for any $z\in X_{t}$ the number
	\begin{align*}
		\lambda:=\sum_{x\in X,x\geq z}\boldsymbol{1}_{D}(x)
	\end{align*}
	is a constant. Each element in a design $D$ is called a \textit{block}.
\end{definition}
Using the language of the incidence operator $W_{i}$, a subset $D\subseteq X$ is a $t$-design of index $\lambda$ if and only if $W_{t}\boldsymbol{1}_{D}=\lambda\boldsymbol{1}_{X_{t}}$. The following lemma is an immediate consequence of the definitions and we include its proof for completeness.
\begin{lemma}\label{lemma_general.simple.bounds.from.definition.of.designs}
	Let $D$ be a $t$-design of index $\lambda$ in $X$. Let $b_{D}$ be the number of blocks in $D$. Then $b_{D}\alpha_{t}=\lambda|X_{t}|$ and for each $z\in X_{i}$, where $i>t$, it holds that $0\leq\left(W_{i}\boldsymbol{1}_{D}\right)(z)=\#\left\{B\in D\mid z\leq B\right\}\leq\lambda$. Consequently, we have
	\begin{align*}
		b_{D}\alpha_{i}\leq\|W_{i}\boldsymbol{1}_{D}\|^{2}=\sum_{z\in X_{i}}\left(\left(W_{i}\boldsymbol{1}_{D}\right)(z)\right)^{2}\leq\lambda b_{D}\alpha_{i}.
	\end{align*}
	\begin{proof}
		For each $z\in X_{i}$, where $i>t$, we choose $y\in X_{t}$ such that $y\leq z$, whose existence is guaranteed by item (2) of Definition \ref{def_G.admissible.graded.poset}). Since every block in $D$ containing $z$ also contains $y$, we conclude that $\left(W_{i}\boldsymbol{1}_{D}\right)(z)\leq\lambda$. It follows that
		\begin{align}\label{ineq_proof.general.simple.bounds.from.definition.of.designs}
			\left(W_{i}\boldsymbol{1}_{D}\right)(z)\leq\left(\left(W_{i}\boldsymbol{1}_{D}\right)(z)\right)^{2}\leq\lambda\left(W_{i}\boldsymbol{1}_{D}\right)(z).
		\end{align}
		Note that a simple counting argument yields $\sum_{z\in X_{i}}\left(W_{i}\boldsymbol{1}_{D}\right)(z)=b_{D}\alpha_{i}$. Therefore, summing the inequalities \eqref{ineq_proof.general.simple.bounds.from.definition.of.designs} for all $z\in X_{i}$ gives the result.
	\end{proof}
\end{lemma}
In the following lemma we study the restriction of the operator $W_{i}^*W_{i}$ on every irreducible component $U_{j}$, which enables us to study its action on the characteristic vector of a design thereafter.
\begin{lemma}\label{lemma_general.spectra.of.gram.matrix.of.inclusion.operator}
	With notation as above, the following statements are equivalent:
	\begin{enumerate}[itemsep=5pt]
		\item[(i)] For each $i\in\left\{0,1,\ldots,m\right\}$ the operator $W_{i}$ satisfies
		\begin{align*}
			\operatorname{Im}(W_{i}^*)=U_{0}\oplus U_{1}\oplus\cdots\oplus U_{i}.
		\end{align*}
		\item[(ii)] There exist scalars $\gamma_{i,j}\geq 0$ such that
		\begin{align*}
			W_{i}^*W_{i}\rvert_{U_{j}}=\gamma_{i,j}I,\qquad\text{$\gamma_{i,j}>0$ for all $0\leq j\leq i$}\qquad\text{and}\qquad\text{$\gamma_{i,j}=0$ for all $i<j\leq m$}.
		\end{align*}
	\end{enumerate}
	\begin{proof}
		(i)$\implies$(ii): Since every $g\in G$ preserves incidence relations, the operators $W_{i}$ and $W_{i}^*$ are both $G$-equivariant. Hence $W_{i}^*W_{i}:\mathbb{C}^{X}\to\mathbb{C}^{X}$ is a $G$-endomorphism. Given that $W_{i}^*W_{i}$ commutes with $G$, it maps every $U_{j}$ into a $G$-submodule of $\mathbb{C}^{X}$. Recall that the decomposition of $\mathbb{C}^{X}$ is multiplicity-free, this implies that $W_{i}^*W_{i}$ preserves $U_{j}$. Therefore, by Schur's lemma we conclude the existence of the scalars $\gamma_{i,j}$ such that $W_{i}^*W_{i}\rvert_{U_{j}}=\gamma_{i,j}I$.\par 
		Note that by the orthogonality of the components $U_{j}$ we have $\ker{W_{i}}=\operatorname{Im}(W_{i}^*)^{\perp}=U_{i+1}\oplus U_{i+2}\oplus\cdots\oplus U_{m}$. This implies $\gamma_{i,j}=0$ if $j>i$. For $j\leq i$, observe that $U_{j}\subseteq\operatorname{Im}(W_{i}^*)$ is orthogonal to $\ker{W_{i}}$, which means for every non-zero vector $u\in U_{j}$ we have $W_{i}u\neq 0$. Therefore
		\begin{equation*}
			\gamma_{i,j}\|u\|^{2}=\langle\gamma_{i,j}u,u\rangle=\langle W_{i}^*W_{i}u,u\rangle=\langle W_{i}u,W_{i}u\rangle=\|W_{i}u\|^{2}>0.
		\end{equation*}\par 
		(ii)$\implies$(i): Note that for all $f\in\mathbb{C}^{X}$ we have $\langle W_{i}^*W_{i}f,f\rangle=\|W_{i}f\|^{2}$. This yields $\ker{W_{i}}=\ker{W_{i}^*W_{i}}=\bigoplus_{\gamma_{i,j}=0}U_{j}$, where we use (ii) to obtain the second equality. Hence $\operatorname{Im}(W_{i}^*)=\left(\ker{W_{i}}\right)^{\perp}=\bigoplus_{\gamma_{i,j}\neq 0}U_{j}=U_{0}\oplus U_{1}\oplus\cdots\oplus U_{i}$, which completes the proof.
	\end{proof}
\end{lemma}
\begin{lemma}\label{lemma_general.vanishing.of.the.first.t.nontrivial.components}
	If either of the items (i) and (ii) in Lemma \ref{lemma_general.spectra.of.gram.matrix.of.inclusion.operator} is satisfied, then for each $t$-design $D$ of index $\lambda$ in $X$, it holds that $\boldsymbol{1}_{D}\perp U_{1}\oplus U_{2}\oplus\cdots\oplus U_{t}$.
	\begin{proof}
		For every $u\in U_{j}$, where $1\leq j\leq t$, by item (ii) of Lemma \ref{lemma_general.spectra.of.gram.matrix.of.inclusion.operator} we have $W_{t}^*W_{t}u=\gamma_{t,j}u$ with $\gamma_{t,j}>0$. Noticing that $W_{t}\boldsymbol{1}_{D}=\lambda\boldsymbol{1}_{X_{t}}$, we compute
		\begin{multline*}
			\langle\boldsymbol{1}_{D},u\rangle=\frac{1}{\gamma_{t,j}}\langle\boldsymbol{1}_{D},W_{t}^*W_{t}u\rangle=\frac{1}{\gamma_{t,j}}\langle W_{t}\boldsymbol{1}_{D},W_{t}u\rangle\\=\frac{\lambda}{\gamma_{t,j}}\langle\boldsymbol{1}_{X_{t}},W_{t}u\rangle=\frac{\lambda}{\gamma_{t,j}}\langle W_{t}^*\boldsymbol{1}_{X_{t}},u\rangle=\frac{\lambda\alpha_{t}}{\gamma_{t,j}}\langle\boldsymbol{1},u\rangle=0,
		\end{multline*}
		where the last equation holds because $U_{j}\perp U_{0}$. This completes the proof.
	\end{proof}
\end{lemma}
With $\gamma_{i,j}$ defined in Lemma \ref{lemma_general.spectra.of.gram.matrix.of.inclusion.operator}, we define
\begin{align*}
	A_{i}:=\frac{\gamma_{i,0}b_{D}}{|X|\alpha_{i}}\quad\text{and}\quad B_{i,j}:=\frac{\lambda\alpha_{j}\gamma_{i,j}}{\alpha_{i}\gamma_{j,j}}
\end{align*}
for all $t<j\leq i\leq m$.
\begin{theorem}\label{thm_general.saturation.theorem}
	Let $G$ be a finite group and let $\left(\Omega,\leq\right)$ be a $G$-admissible graded finite poset with a fixed rank $m$. Fix integer $t$ such that $0\leq t<m$. Let $X$ be the top fiber of $\Omega$ and let $D$ be a $t$-design of index $\lambda$ in $X$. Let $W_{i}$ be the incidence operator defined as previously. Assume that $G$ and $\Omega$ satisfy the following conditions:
	\begin{enumerate}[itemsep=5pt]
		\item The Schurian scheme of $G$ on $X$ is commutative;
		\item Either of the items (i) and (ii) in Lemma \ref{lemma_general.spectra.of.gram.matrix.of.inclusion.operator} holds;
		\item For every $i=t+1,t+2,\ldots,m$, it holds that $A_{i}+\sum_{j=t+1}^{i-1}B_{i,j}(1-\lambda^{-1}A_{j})<1$.
	\end{enumerate}
	Then we have
	\begin{align*}
		\operatorname{proj}_{U_{1}}(\boldsymbol{1}_{D})=\operatorname{proj}_{U_{2}}(\boldsymbol{1}_{D})=\cdots=\operatorname{proj}_{U_{t}}(\boldsymbol{1}_{D})=0
	\end{align*}
	and $\operatorname{proj}_{U_{i}}(\boldsymbol{1}_{D})\neq 0$ for every $i\in\left\{0,t+1,t+2,\ldots,m\right\}$. More precisely, we have $\|\operatorname{proj}_{U_{0}}(\boldsymbol{1}_{D})\|^{2}=b_{D}^{2}/|X|$ and
	\begin{equation*}
		\|\operatorname{proj}_{U_{i}}(\boldsymbol{1}_{D})\|^{2}\geq\frac{b_{D}\alpha_{i}}{\gamma_{i,i}}\left(1-A_{i}-\sum_{j=t+1}^{i-1}B_{i,j}\left(1-\frac{A_{j}}{\lambda}\right)\right)>0
	\end{equation*}
	for every $i\in\left\{t+1,t+2,\ldots,m\right\}$.
	\begin{proof}
		Let $\boldsymbol{1}_{D}=f_{0}+f_{1}+\cdots+f_{m}$, $f_{j}\in U_{j}$ be the orthogonal decomposition of $\boldsymbol{1}_{D}\in\mathbb{C}^{X}$ and put $p_{j}:=\|f_{j}\|^{2}$ for each $0\leq j\leq m$. By Lemma \ref{lemma_general.vanishing.of.the.first.t.nontrivial.components}, we find that
		\begin{align*}
			p_{j}=0\quad\text{for all $1\leq j\leq t$}.
		\end{align*}
		Since $f_{0}=c\boldsymbol{1}$ for some $c\in\mathbb{C}$, by $\langle\boldsymbol{1}_{D}-c\boldsymbol{1},\boldsymbol{1}\rangle=\langle\boldsymbol{1}_{D}-f_{0},\boldsymbol{1}\rangle=0$ we obtain that $c=\langle\boldsymbol{1}_{D},\boldsymbol{1}\rangle/\langle\boldsymbol{1},\boldsymbol{1}\rangle$. Therefore,
		\begin{align*}
			f_{0}=\frac{\langle\boldsymbol{1}_{D},\boldsymbol{1}\rangle}{\langle\boldsymbol{1},\boldsymbol{1}\rangle}\boldsymbol{1}=\frac{\sum_{x\in X}\boldsymbol{1}_{D}(x)}{|X|}\boldsymbol{1}=\frac{b_{D}}{|X|}\boldsymbol{1},
		\end{align*}
		which implies $p_{0}=\|f_{0}\|^{2}=b_{D}^{2}/|X|$.\par 
		We prove the lower bound on $p_{i}$ for every $i\in\left\{t+1,t+2,\ldots,m\right\}$. By Lemma \ref{lemma_general.spectra.of.gram.matrix.of.inclusion.operator}, we have
		\begin{multline}\label{eq_general.inclusion.norm.expansion}
			S_{i}:=\|W_{i}\boldsymbol{1}_{D}\|^{2}=\langle W_{i}\boldsymbol{1}_{D},W_{i}\boldsymbol{1}_{D}\rangle=\langle\boldsymbol{1}_{D},W_{i}^*W_{i}\boldsymbol{1}_{D}\rangle\\=\sum_{j=0}^{m}\sum_{\ell=0}^{m}\langle f_{\ell},\gamma_{i,j}f_{j}\rangle=\sum_{j=0}^{m}\sum_{\ell=0}^{m}\gamma_{i,j}\langle f_{\ell},f_{j}\rangle=\gamma_{i,0}p_{0}+\sum_{j=t+1}^{i}\gamma_{i,j}p_{j}.
		\end{multline}
		For contradiction, we assume
		\begin{equation}\label{eq_general.saturation.contradiction.assumption}
			p_{i}<\frac{b_{D}\alpha_{i}}{\gamma_{i,i}}\left(1-A_{i}-\sum_{j=t+1}^{i-1}B_{i,j}\left(1-\frac{A_{j}}{\lambda}\right)\right)
		\end{equation}
		for some $i>t$. By Lemma \ref{lemma_general.simple.bounds.from.definition.of.designs} and the non-negativity of $\gamma_{i,j}$ and $p_{j}$, it holds that
		\begin{align*}
			p_{\ell}\leq\frac{S_{\ell}-\gamma_{\ell,0}p_{0}}{\gamma_{\ell,\ell}}\leq\frac{b_{D}\alpha_{\ell}(\lambda-A_{\ell})}{\gamma_{\ell,\ell}}\quad\text{for all $t+1\leq\ell<i$}.
		\end{align*}
		Applying this to all $\gamma_{i,j}p_{j}$ in \eqref{eq_general.inclusion.norm.expansion} with $t+1\leq j<i$ gives
		\begin{multline*}
			S_{i}=\gamma_{i,0}p_{0}+\sum_{j=t+1}^{i-1}\gamma_{i,j}p_{j}+\gamma_{i,i}p_{i}
			\leq\gamma_{i,0}p_{0}+\sum_{j=t+1}^{i-1}\gamma_{i,j}\frac{b_{D}\alpha_{j}(\lambda-A_{j})}{\gamma_{j,j}}+\gamma_{i,i}p_{i}\\
			=b_{D}\alpha_{i}\left(A_{i}+\sum_{j=t+1}^{i-1}B_{i,j}\left(1-\frac{A_{j}}{\lambda}\right)+\frac{\gamma_{i,i}}{b_{D}\alpha_{i}}p_{i}\right)
			<b_{D}\alpha_{i},
		\end{multline*}
		where we use the assumption \eqref{eq_general.saturation.contradiction.assumption} to obtain the last inequality. This contradicts the lower bound $S_{i}\geq b_{D}\alpha_{i}$ given by Lemma \ref{lemma_general.simple.bounds.from.definition.of.designs}. Note that when $i=t+1$ the sum above is empty, so the same argument applies. This completes the proof.
	\end{proof}
\end{theorem}
\begin{remark}
	One can see that $A_{i}=\gamma_{i,0}p_{0}/(b_{D}\alpha_{i})\leq S_{i}/(b_{D}\alpha_{i})\leq\lambda$, hence the stronger defect condition $A_{i}+\sum_{j=t+1}^{i-1}B_{i,j}<1$ is also applicable since it automatically implies the condition (3). For convenience of computation, we use this condition instead when applying Theorem \ref{thm_general.saturation.theorem} in the sequel.
\end{remark}
We also mention that the defect condition $A_{t+1}<1$ in Theorem \ref{thm_general.saturation.theorem} implies that $D$ is not a $(t+1)$-design. This is also a hypothesis of the previous result \cite[Theorem 4]{ghorbani2025vector} about the dimension of the space spanned by permutations of a combinatorial $t$-design. Indeed, if we assume that $D$ is already a $(t+1)$-design, then $p_{t+1}=0$ would yield $S_{t+1}=b_{D}\alpha_{t+1}A_{t+1}<b_{D}\alpha_{t+1}$, which is a contradiction.\par 
Using Theorem \ref{thm_general.saturation.theorem} and our orbit span lemma, we obtain the following result about the spanning space of the orbit of a fixed design $D$.
\begin{corollary}\label{coro_general.saturation.and.dimension}
	With notation as above, we assume that the conditions in Theorem \ref{thm_general.saturation.theorem} are all satisfied. Then we have
	\begin{align*}
		\operatorname{span}_{\mathbb{C}}\left\{g\boldsymbol{1}_{D}\mid g\in G\right\}=U_{0}\oplus U_{t+1}\oplus U_{t+2}\oplus\cdots\oplus U_{m}.
	\end{align*}
	In particular,
	\begin{align*}
		\dim_{\mathbb{Q}}\operatorname{span}_{\mathbb{Q}}\left\{g\boldsymbol{1}_{D}\mid g\in G\right\}=|X|-\sum_{j=1}^{t}\dim{U_{j}}.
	\end{align*}
	\begin{proof}
		By Theorem \ref{thm_general.saturation.theorem}, the characteristic vector $\boldsymbol{1}_{D}$ has non-zero orthogonal projections exactly on the irreducible components $U_{0},U_{t+1},U_{t+2},\ldots,U_{m}$. By Lemma \ref{lemma_orbit.span.lemma}, we conclude that $\operatorname{span}_{\mathbb{C}}\left\{g\boldsymbol{1}_{D}\mid g\in G\right\}=U_{0}\oplus U_{t+1}\oplus U_{t+2}\oplus\cdots\oplus U_{m}$, which completes the proof.
	\end{proof}
\end{corollary}
\begin{remark}
	Since $\boldsymbol{1}_{D}$ and all of its $G$-translates have rational coordinates, the dimensions of their spans over $\mathbb{Q}$ and $\mathbb{C}$ are the same. Indeed, if we form a matrix whose columns are the vectors $g\boldsymbol{1}_{D}$, $g\in G$, then this matrix has rational entries, and its rank is unchanged whether it is regarded as a matrix over $\mathbb{Q}$ or over $\mathbb{C}$. Therefore,
	\begin{align*}
		\dim_{\mathbb{Q}}
		\operatorname{span}_{\mathbb{Q}}
		\left\{g\boldsymbol{1}_{D}\mid g\in G\right\}
		=
		\dim_{\mathbb{C}}
		\operatorname{span}_{\mathbb{C}}
		\left\{g\boldsymbol{1}_{D}\mid g\in G\right\}.
	\end{align*}
	In Corollary \ref{coro_general.saturation.and.dimension} and what follows, we state the dimension of the orbit span over $\mathbb{Q}$ in order to be consistent with the corresponding dimension results in the previous literature.
\end{remark}

\section{Asymptotic saturation theorems in the Hamming scheme and the bilinear forms scheme}\label{sect_hamming.and.bilinear.forms}
The Hamming scheme and the bilinear forms scheme are both translation association schemes whose underlying sets $X$ have natural abelian group structures, in the sense that for all classes $R$ of relations, $(x,y)\in R$ implies $(x+z,y+z)\in R$ for all $x,y,z\in X$ (see \cite[\S 2.10]{brouwer1989distance}). This fact allows us to use the Fourier decomposition of the irreducible components of $\mathbb{C}^{X}$ to characterize the action of the incidence operators on these components.
\subsection{The Hamming scheme}
Fix integers $0\leq t<n$. Let $\lambda\geq 1$ and let $q$ be a prime power. Let $A=\mathbb{F}_{q}$ be the alphabet and let $X=A^{n}$. Let $\Omega=\left\{v\in A^{I}\mid I\subseteq[n]\right\}$. If we define the rank of each $v\in A^{I}$ to be the cardinality $\# I$ and equip $\Omega$ with the partial order $a\leq x$ if and only if $x\rvert_{I}=a$ for some $I\subseteq[n]$, then $\left(\Omega,\leq\right)$ is a $\mathfrak{S}_{q}\wr\mathfrak{S}_{n}$-admissible graded poset and $X$ is the top fiber of $\Omega$.\par 
A $t$-design of index $\lambda$ in $X$ is therefore a subset $D\subseteq X$ such that in any fixed $t$ positions, every $t$-tuple in $A^{t}$ appears exactly $\lambda$ times. Here $D$ is also known as an \textit{orthogonal array} of strength $t$. It is straightforward to see from its combinatorial meaning that the number of blocks in $D$ equals $b_{D}=\lambda q^{t}$. Therefore, the inequalities in Lemma \ref{lemma_general.simple.bounds.from.definition.of.designs} become
\begin{align*}
	\lambda q^{t}{n\choose i}\leq\|W_{i}\boldsymbol{1}_{D}\|^{2}\leq\lambda^{2}q^{t}{n\choose i}.
\end{align*}\par 
Let $\psi:A\to\mathbb{C}^{\times}$ be a non-trivial additive character. For each $\eta=(\eta_{1},\eta_{2},\ldots,\eta_{n})\in X$, define
\begin{align*}
	\chi_{\eta}:X&\to\mathbb{C}^{\times}\\
	x&\mapsto\psi\left(\langle\eta,x\rangle\right).
\end{align*}
Then each $\chi_{\eta}$ is a character on $X$. We define $\operatorname{wt}(\chi_{\eta})=\operatorname{wt}(\eta)=\#\left\{i\mid\eta_{i}\neq 0\right\}$. The following result is well known and straightforward, which shows that the $j$-th irreducible component $U_{j}$ is spanned by the characters of weight $j$.
\begin{lemma}\label{lemma_hamming.multiplicity.free.decomposition}
	With notation as above, we have an orthogonal multiplicity-free decomposition
	\begin{equation*}
		\mathbb{C}^{X}\cong U_{0}\oplus U_{1}\oplus\cdots\oplus U_{n},
	\end{equation*}
	where $U_{j}=\operatorname{span}_{\mathbb{C}}\left\{\chi_{\eta}\mid\eta\in X,\operatorname{wt}(\eta)=j\right\}$ and we have $\dim{U_{j}}={n\choose j}(q-1)^{j}$ for each $j\in\left\{0,1,\ldots,n\right\}$.
\end{lemma}
\begin{lemma}\label{lemma_hamming.spectra.of.gram.matrix.of.inclusion.operator}
	With notation as above, for $0\leq i\leq n$,
	\begin{align}\label{eq_hamming.decomposition.of.the.image.of.adjoint.inclusion.operator}
		\operatorname{Im}\left(W_{i}^*\right)=\operatorname{span}_{\mathbb{C}}\left\{\boldsymbol{1}_{\left\{x\in X\mid x\rvert_{I}=a\right\}}\mid|I|=i,a\in A^{I}\right\}=U_{0}\oplus U_{1}\oplus\cdots\oplus U_{i}.
	\end{align}
	Furthermore, for $0\leq j\leq i\leq n$, the operator $W_{i}^*W_{i}$ acts on $U_{j}$ as multiplication by
	\begin{align*}
		\gamma_{i,j}(q)=q^{n-i}{n-j\choose i-j}
	\end{align*}
	whereas for $j>i$ the operator $W_{i}$ vanishes on $U_{j}$.
	\begin{proof}
		We first prove \eqref{eq_hamming.decomposition.of.the.image.of.adjoint.inclusion.operator}. For each $I\subseteq[n]$ we consider the subspace
		\begin{align*}
			R_{I}=\operatorname{span}_{\mathbb{C}}\left\{\boldsymbol{1}_{\left\{x\in X\mid x\rvert_{I}=a\right\}}\mid a\in A^{I}\right\}.
		\end{align*}
		Note that each set $\left\{x\in X\mid x\rvert_{I}=a\right\}$ is a coset of $K_{I}:=\left\{x\in X\mid x\rvert_{I}=0\right\}$. Therefore, we have $R_{I}=\operatorname{span}_{\mathbb{C}}\left\{\chi_{\eta}\mid\eta\in X,\chi_{\eta}(x)=1~\text{for all}~x\in K_{I}\right\}$. Since $\chi_{\eta}(x)=1$ for all $x\in K_{I}$ if and only if $\operatorname{supp}(\eta)\subseteq I$, we conclude that $R_{I}=\operatorname{span}_{\mathbb{C}}\left\{\chi_{\eta}\mid\eta\in X,\operatorname{supp}(\eta)\subseteq I\right\}$. This implies
		\begin{align*}
			\operatorname{Im}\left(W_{i}^*\right)=\sum_{I\subseteq[n]\atop|I|=i}R_{I}=\operatorname{span}_{\mathbb{C}}\left\{\chi_{\eta}\mid\eta\in X,\operatorname{wt}(\eta)\leq i\right\}=\bigoplus_{j=0}^{i}U_{j},
		\end{align*}
		giving \eqref{eq_hamming.decomposition.of.the.image.of.adjoint.inclusion.operator}.\par 
		For the second part, note that by Lemma \ref{lemma_hamming.multiplicity.free.decomposition} it suffices to consider $\chi_{\eta}$ with $\operatorname{wt}(\eta)=j$. It follows that
		\begin{align*}
			\left(W_{i}^*W_{i}\chi_{\eta}\right)(x)=\sum_{I\subseteq[n]\atop|I|=i}\sum_{x'\in X\atop x'\rvert_{I}=x\rvert_{I}}\chi_{\eta}(x')=\sum_{I\subseteq[n]\atop|I|=i}\sum_{x''\in K_{I}}\chi_{\eta}(x)\chi_{\eta}(x'').
		\end{align*}
		By the orthogonality of characters, the sum $\sum_{x''\in K_{I}}\chi_{\eta}(x'')$ is non-zero if and only if $\chi_{\eta}$ is trivial on $K_{I}$, which means $\operatorname{supp}(\eta)\subseteq I$. Hence
		\begin{align*}
			\left(W_{i}^*W_{i}\chi_{\eta}\right)(x)=\sum_{\operatorname{supp}(\eta)\subseteq I\subseteq[n]\atop|I|=i}\chi_{\eta}(x)q^{n-i}=q^{n-i}{n-j\choose i-j}\chi_{\eta}(x)
		\end{align*}
		for all $0\leq j\leq i$. Meanwhile, $W_{i}\chi_{\eta}$ vanishes if $j>i$ since there does not exist an $i$-element subset $I\subseteq[n]$ such that $\operatorname{supp}(\eta)\subseteq I$. This completes the proof.
	\end{proof}
\end{lemma}
\begin{lemma}\label{lemma_hamming.estimates.on.A_{i}(n).and.B_{i,j}(n)}
	We write
	\begin{align*}
		A_{i}(q):=\frac{b_{D}\gamma_{i,0}(q)}{|X|\alpha_{i}}\quad\text{and}\quad B_{i,j}(q):=\lambda\frac{\alpha_{j}\gamma_{i,j}(q)}{\alpha_{i}\gamma_{j,j}(q)}.
	\end{align*}
	Then there exists $q_{0}:=q_{0}(n,t,\lambda)$ such that for every $i\in\left\{t+1,t+2,\ldots,n\right\}$ the inequality $A_{i}(q)+\sum_{j=t+1}^{i-1}B_{i,j}(q)<1$ holds for all $q\geq q_{0}$.
	\begin{proof}
		Note that
		\begin{equation*}
			A_{i}(q)=\frac{\lambda q^{t+n-i}{n\choose i}}{q^{n}{n\choose i}}=\lambda q^{-(i-t)}=o_{q}(1)
		\end{equation*}
		and
		\begin{equation*}
			B_{i,j}(q)=\lambda\frac{q^{n-i}{n-j\choose i-j}{n\choose j}}{q^{n-j}{n\choose i}}=\lambda{i\choose j}q^{-(i-j)}=o_{q}(1).\qedhere
		\end{equation*}
	\end{proof}
\end{lemma}
Lemmas \ref{lemma_hamming.multiplicity.free.decomposition}, \ref{lemma_hamming.spectra.of.gram.matrix.of.inclusion.operator}, and \ref{lemma_hamming.estimates.on.A_{i}(n).and.B_{i,j}(n)} verify all the hypotheses of Theorem \ref{thm_general.saturation.theorem} for the Hamming scheme when the alphabet size $q$ is sufficiently large. Therefore, we obtain the following result.
\begin{theorem}\label{thm_hamming.asymptotic.saturation}
	Fix $n,t,\lambda$. There exists $q_{0}=q_{0}(n,t,\lambda)$ such that, for every prime power $q\geq q_{0}$ and every $t$-$(q,n,\lambda)$ orthogonal array $D$, it holds that
	\begin{align*}
		\operatorname{span}_{\mathbb{C}}\left\{g\boldsymbol{1}_{D}\mid g\in G\right\}=U_{0}\oplus U_{t+1}\oplus U_{t+2}\oplus\cdots\oplus U_{n},
	\end{align*}
	where $G=\mathfrak{S}_{q}\wr\mathfrak{S}_{n}$ and each $U_{j}$ is as described in Lemma \ref{lemma_hamming.multiplicity.free.decomposition}. In particular,
	\begin{align*}
		\dim_{\mathbb{Q}}\operatorname{span}_{\mathbb{Q}}\left\{g\boldsymbol{1}_{D}\mid g\in G\right\}=q^{n}-\sum_{j=1}^{t}{n\choose j}(q-1)^{j}.
	\end{align*}
\end{theorem}
We remark that, since $q$ is a prime power, the existence of an orthogonal array with a large alphabet size $q$ is well known. This is because the orthogonal arrays of index one are the same as \textit{maximum distance separable (MDS) codes} \cite[Theorem 4.21]{hedayat1999orthogonal}, and there exists an MDS code of length $n$ and with alphabet size $q$ once $n\leq q$ via the use of \textit{(generalized) Reed--Solomon codes} \cite[Proposition 5.1]{roth2006introduction}. Therefore, using the fact that the Hamming scheme is translational, we obtain that the union of any $\lambda\leq q^{n-t}$ distinct cosets of a $t$-$(q,n,1)$ orthogonal array in $\mathbb{F}_{q}^{n}$ forms a $t$-$(q,n,\lambda)$ orthogonal array.\par 
Moreover, although it is possible to consider orthogonal arrays over an arbitrary alphabet of size $q$, the existence of orthogonal arrays with prescribed parameters of arbitrary $q$ is much more delicate, and the available general existence results do not appear to guarantee examples in the parameter range required by Theorem \ref{thm_general.saturation.theorem}. Indeed, it was proved by Ray-Chaudhuri and Singhi \cite{ray1988existence} in 1988 that a $t$-$(q,n,\lambda)$ orthogonal array exists whenever
\begin{align}\label{ineq_existence.of.orthogonal.array.raychaudhuri-singhi}
	\lambda\geq F(q,n,t)
\end{align}
for some function $F$. In 2017, Kuperberg, Lovett and Peled \cite[Theorem 1.2]{kuperberg2017probabilistic} proved that there exists at least one $t$-$(q,n,\lambda)$ orthogonal array with index $\lambda$ in the central range
\begin{align}\label{ineq_existence.of.orthogonal.array.kuperberg-lovett-peled}
	\left(\frac{cn}{t}\right)^{2ct}q^{(2c-1)t}\leq\lambda\leq q^{n-t}-\left(\frac{cn}{t}\right)^{2ct}q^{(2c-1)t},
\end{align}
where $c>0$ is a universal constant. From the proof of Lemma \ref{lemma_hamming.estimates.on.A_{i}(n).and.B_{i,j}(n)} we see that our inequality $A_{i}(q)+\sum_{j=t+1}^{i-1}B_{i,j}(q)<1$ can hold only if $\lambda<q$, which is not obviously compatible with either of the conditions \eqref{ineq_existence.of.orthogonal.array.raychaudhuri-singhi} and \eqref{ineq_existence.of.orthogonal.array.kuperberg-lovett-peled}.
\subsection{The bilinear forms scheme}
Let $q$ be a prime power. Let $W=\mathbb{F}_{q}^{m}$, $U=\mathbb{F}_{q}^{m'}$, where $m\leq m'$, and let $X=\operatorname{Hom}(W,U)$. Since homomorphisms in $X$ can be naturally viewed as the bilinear forms over $W\times U$, we treat $X$ under the bilinear forms scheme $\operatorname{Bil}_{q}(m,m')$ (see \cite{delsarte1978bilinear} and \cite[\S 6.4]{bannai2021algebraic}). Let $\Omega=\left\{f\in\operatorname{Hom}(S,U)\mid S\subseteq W\right\}$ and let the rank of each $f\in\operatorname{Hom}(S,U)$ be $\dim{S}$. For $g,h\in\Omega$, we write $g\leq h$ if there exists some subspace $S\subseteq W$ such that $g=h\rvert_{S}$. Then $\leq$ is a partial order on $\Omega$ and $\left(\Omega,\leq\right)$ is a $\mathbb{F}_{q}^{m'\times m}\rtimes\left(\operatorname{GL}(m,q)\times\operatorname{GL}(m',q)\right)$-admissible graded poset and $X$ is its top fiber.\par 
Designs in the bilinear forms scheme are often viewed as $q$-analogs of orthogonal arrays. Indeed, a $t$-design $D$ of index $\lambda$ in $X$ satisfies that for each $t$-dimensional subspace $W_{0}$ of $W$ and for each $f_{0}\in\operatorname{Hom}(W_{0},U)$, there are exactly $\lambda$ homomorphisms $f\in D$ such that $f\rvert_{W_{0}}=f_{0}$. It is straightforward to see that the number of blocks in $D$ equals $b_{D}=\lambda q^{tm'}$. Therefore, from Lemma \ref{lemma_general.simple.bounds.from.definition.of.designs} we derive
\begin{align*}
	\lambda q^{tm'}{m\brack i}_{q}\leq\|W_{i}\boldsymbol{1}_{D}\|^{2}\leq\lambda^{2}q^{tm'}{m\brack i}_{q}.
\end{align*}\par 
Let $\psi:\mathbb{F}_{q}\to\mathbb{C}^{\times}$ be a non-trivial additive character. For each $C\in X$, define
\begin{align*}
	\chi_{C}:X&\to\mathbb{C}^{\times}\\
	A&\mapsto\psi\left(\operatorname{Tr}(C^{T}A)\right).
\end{align*}
Then each $\chi_{C}$ is a character on $X$. The following result is well known and straightforward, where the dimension of each $U_{j}$ (which can be found in \cite[p. 357]{bannai2021algebraic}) is equal to the number of $m'\times m$ matrices of rank $j$ over $\mathbb{F}_{q}$.
\begin{lemma}\label{lemma_bilinear.forms.multiplicity.free.decomposition}
	With notation as above, the characters $\chi_{C}$, $C\in X$ form an orthogonal basis of $\mathbb{C}^{X}$. Furthermore, we have an orthogonal multiplicity-free decomposition
	\begin{align*}
		\mathbb{C}^{X}=U_{0}\oplus U_{1}\oplus\cdots\oplus U_{m},
	\end{align*}
	where $U_{j}=\operatorname{span}_{\mathbb{C}}\left\{\chi_{C}\mid C\in X,\operatorname{rank}(C)=j\right\}$ and $\dim{U_{j}}={m\brack j}_{q}(q^{m'}-1)(q^{m'-1}-1)\cdots(q^{m'-j+1}-1)q^{{j\choose 2}}$ for each $j\in\left\{0,1,\ldots,m\right\}$.
\end{lemma}
\begin{lemma}\label{lemma_bilinear.forms.the.vanishing.functions.and.trivial.characters}
	Let $S$ be a subspace of $W$ and define $K_{S}:=\left\{F\in X\mid F\rvert_{S}=0\right\}$. Then for all $B\in X$, the character $\chi_{B}$ is trivial on $K_{S}$ (i.e. $\chi_{B}(F)=1$ holds for all $F\in K_{S}$) if and only if $\operatorname{Im}(B^{T})\subseteq S$.
	\begin{proof}
		The ``if'' part is straightforward. Conversely, suppose $B\in X$ which defines a trivial character $\chi_{B}$ on $K_{S}$ and assume that $\operatorname{Im}(B^{T})\not\subseteq S$ for contradiction. Then there exists a $v_{0}\in U$ such that $B^{T}v_{0}\notin S$, in other words, $\pi_{S}(B^{T}v_{0})\neq 0$ holds for the projection $\pi_{S}:W\to W/S$. Take $\alpha$ to be a linear functional on $W/S$ satisfying $\alpha\left(\pi_{S}B^{T}v_{0}\right)\neq 0$ (such $\alpha$ exists because $\pi_{S}B^{T}v_{0}\in W/S$ is non-zero). We define the rank-one map
		\begin{align*}
			\phi:W/S&\to U\\
			u&\mapsto\alpha(u)v_{0}.
		\end{align*}
		Then by the linearity of $\alpha$, the composition map $\phi\pi_{S}:W\to U$ vanishes on $S$, which means $\phi\pi_{S}\in K_{S}$. If we write $v_{0}=\sum_{i=1}^{m'}c_{i}e_{i}$, where $c_{i}\in\mathbb{F}_{q}$ and $\left\{e_{i}\right\}_{1\leq i\leq m'}$ is a basis of $U$, and denote by $d_{i}$ the value $\alpha(\pi_{S}B^{T}e_{i})$ for each $i$, then we compute
		\begin{align*}
			\left(\phi\pi_{S}B^{T}\right)(e_{1},e_{2},\ldots,e_{m'})=(d_{1}v_{0},d_{2}v_{0},\ldots,d_{m'}v_{0})=(d_{j}c_{i})_{i,j}(e_{1},e_{2},\ldots,e_{m'})^{T},
		\end{align*}
		which yields
		\begin{align*}
			\operatorname{Tr}(B^{T}\phi\pi_{S})=\operatorname{Tr}(\phi\pi_{S}B^{T})=\sum_{i=1}^{m'}c_{i}d_{i}=\alpha(\pi_{S}B^{T}v_{0})\neq 0.
		\end{align*}
		Hence, there exists a scalar $c\in\mathbb{F}_{q}^{\times}$ so that $\chi_{B}(c\phi\pi_{S})\neq 1$, which contradicts the definition of $B$. This completes the proof.
	\end{proof}
\end{lemma}
\begin{lemma}\label{lemma_bilinear.forms.spectra.of.gram.matrix.of.inclusion.operator}
	With notation as above, for $0\leq i\leq m$,
	\begin{align}\label{eq_bilinear.forms.decomposition.of.the.image.of.adjoint.inclusion.operator}
		\operatorname{Im}\left(W_{i}^*\right)=\operatorname{span}_{\mathbb{C}}\left\{\boldsymbol{1}_{\left\{F\in X\mid F\rvert_{S}=\varphi\right\}}\mid S\subseteq W,\dim{S}=i,\varphi\in\operatorname{Hom}(S,U)\right\}=U_{0}\oplus U_{1}\oplus\cdots\oplus U_{i}.
	\end{align}
	Furthermore, for $0\leq j\leq i\leq m$, the operator $W_{i}^*W_{i}$ acts on $U_{j}$ as multiplication by
	\begin{align*}
		\gamma_{i,j}(m')=q^{m'(m-i)}{m-j\brack i-j}_{q}
	\end{align*}
	whereas for $j>i$ the operator $W_{i}$ vanishes on $U_{j}$.
	\begin{proof}
		For each subspace $S\subseteq W$, define
		\begin{align*}
			R_{S}=\operatorname{span}_{\mathbb{C}}\left\{\boldsymbol{1}_{\left\{F\in X\mid F\rvert_{S}=\varphi\right\}}\mid\varphi\in\operatorname{Hom}(S,U)\right\}.
		\end{align*}
		Note that each set $\left\{F\in X\mid F\rvert_{S}=\varphi\right\}$ is a coset of $K_{S}$. It follows that $R_{S}$ can be identified with the space of functions on $X$ that are constant on the cosets of $K_{S}$. In other words,
		\begin{align*}
			R_{S}=\operatorname{span}_{\mathbb{C}}\left\{\chi_{B}\mid B\in X,\chi_{B}(F)=1~\text{for all $F\in K_{S}$}\right\}.
		\end{align*}
		By Lemma \ref{lemma_bilinear.forms.the.vanishing.functions.and.trivial.characters}, we have $R_{S}=\operatorname{span}_{\mathbb{C}}\left\{\chi_{B}\mid\operatorname{Im}(B^{T})\subseteq S\right\}$. This yields
		\begin{align*}
			\operatorname{Im}(W_{i}^*)=\sum_{S\subseteq W\atop\dim{S}=i}R_{S}=\operatorname{span}_{\mathbb{C}}\left\{\chi_{B}\mid\operatorname{rank}(B)\leq i\right\}=\bigoplus_{j=0}^{i}U_{j},
		\end{align*}
		which gives \eqref{eq_bilinear.forms.decomposition.of.the.image.of.adjoint.inclusion.operator}.\par 
		For the second part, recall that by Lemma \ref{lemma_bilinear.forms.multiplicity.free.decomposition} it suffices to consider the characters $\chi_{C}$ with $\operatorname{rank}(C)=j$. It follows that
		\begin{align*}
			\left(W_{i}^*W_{i}\chi_{C}\right)(f)=\sum_{S\subseteq W\atop\dim{S}=i}\sum_{g\in X\atop g\rvert_{S}=f\rvert_{S}}\chi_{C}(g)=\sum_{S\subseteq W\atop\dim{S}=i}\sum_{h\in K_{S}}\chi_{C}(f)\chi_{C}(h).
		\end{align*}
		By the orthogonality of characters, the sum $\sum_{h\in K_{S}}\chi_{C}(h)$ is non-zero if and only if $\chi_{C}$ is the trivial character on $K_{S}$, which is, by Lemma \ref{lemma_bilinear.forms.the.vanishing.functions.and.trivial.characters}, equivalent to $\operatorname{Im}(C^{T})\subseteq S$. This implies
		\begin{align*}
			\left(W_{i}^*W_{i}\chi_{C}\right)(f)=\sum_{\operatorname{Im}(C^{T})\subseteq S\subseteq W\atop\dim{S}=i}|K_{S}|\chi_{C}(f)=q^{m'(m-i)}{m-j\brack i-j}_{q}\chi_{C}(f)
		\end{align*}
		for all $0\leq j\leq i$, where we use the fact that $|K_{S}|=q^{m'(m-i)}$. For $j>i$ the sum vanishes since there does not exist an $i$-dimensional subspace $S$ containing $\operatorname{Im}(C^{T})$. This completes the proof.
	\end{proof}
\end{lemma}
\begin{lemma}\label{lemma_bilinear.forms.estimates.on.A_{i}(n).and.B_{i,j}(n)}
	We write
	\begin{align*}
		A_{i}(m'):=\frac{b_{D}\gamma_{i,0}(m')}{|X|\alpha_{i}}\quad\text{and}\quad B_{i,j}(m'):=\lambda\frac{\alpha_{j}\gamma_{i,j}(m')}{\alpha_{i}\gamma_{j,j}(m')}.
	\end{align*}
	Then there exists $m'_{0}:=m'_{0}(m,q,t,\lambda)$ such that for every $i\in\left\{t+1,t+2,\ldots,m\right\}$ the inequality $A_{i}(m')+\sum_{j=t+1}^{i-1}B_{i,j}(m')<1$ holds for all $m'\geq m'_{0}$.
	\begin{proof}
		Note that
		\begin{equation*}
			A_{i}(m')=\frac{\lambda q^{tm'+m'(m-i)}{m\brack i}_{q}}{q^{mm'}{m\brack i}_{q}}=\lambda q^{-m'(i-t)}=\Theta(q^{-(i-t)m'})
		\end{equation*}
		and
		\begin{equation*}
			B_{i,j}(m')=\lambda\frac{q^{m'(m-i)}{m-j\brack i-j}_{q}{m\brack j}_{q}}{q^{m'(m-j)}{m\brack j}_{q}}=\lambda{i\brack j}_{q}q^{-m'(i-j)}=\Theta(q^{-(m'-j)(i-j)}).\qedhere
		\end{equation*}
	\end{proof}
\end{lemma}
Hence, Lemmas \ref{lemma_bilinear.forms.multiplicity.free.decomposition}, \ref{lemma_bilinear.forms.spectra.of.gram.matrix.of.inclusion.operator}, and \ref{lemma_bilinear.forms.estimates.on.A_{i}(n).and.B_{i,j}(n)} verify all the hypotheses of Theorem \ref{thm_general.saturation.theorem} for the bilinear forms scheme when $m'$ is sufficiently large.
\begin{theorem}\label{thm_bilinear.forms.asymptotic.saturation}
	Fix $m,t,q,\lambda$. There exists a constant $m'_{0}=m'_{0}(m,t,q,\lambda)$ such that, for every $m'\geq m'_{0}$ and every $t$-design $D$ of index $\lambda$ in the bilinear forms scheme $\operatorname{Bil}_{q}(m,m')$, it holds that
	\begin{align*}
		\operatorname{span}_{\mathbb{C}}\left\{g\boldsymbol{1}_{D}\mid g\in G\right\}=U_{0}\oplus U_{t+1}\oplus U_{t+2}\oplus\cdots\oplus U_{m},
	\end{align*}
	where $G=\mathbb{F}_{q}^{m'\times m}\rtimes\left(\operatorname{GL}(m,q)\times\operatorname{GL}(m',q)\right)$ and each $U_{j}$ is as described in Lemma \ref{lemma_bilinear.forms.multiplicity.free.decomposition}. In particular,
	\begin{align*}
		\dim_{\mathbb{Q}}\operatorname{span}_{\mathbb{Q}}\left\{g\boldsymbol{1}_{D}\mid g\in G\right\}=q^{mm'}-\sum_{j=1}^{t}{m\brack j}_{q}(q^{m'}-1)(q^{m'-1}-1)\cdots(q^{m'-j+1}-1)q^{{j\choose 2}}.
	\end{align*}
\end{theorem}
We note that a $t$-design of index one in $\operatorname{Bil}_{q}(m,m')$ is equivalent to an \textit{$(m\times m',t)_{q}$-maximum rank distance (MRD) code}, which does exist for all relevant parameters \cite{delsarte1978bilinear}. Therefore, taking the union of any $\lambda$ cosets of an $(m\times m',t)_{q}$-MRD code in $\mathbb{F}_{q}^{m\times m'}$ we obtain a $t$-design of index $\lambda$ in $\operatorname{Bil}_{q}(m,m')$. In an earlier result of the authors \cite[Theorem 1.2]{li2026dimension.mrd}, it is shown that the dimension in Theorem \ref{thm_bilinear.forms.asymptotic.saturation} is attained by the space generated by the characteristic vectors of $(m\times m',t)_{q}$-MRD codes. The result of Theorem \ref{thm_bilinear.forms.asymptotic.saturation} extends this to the orbit of any single $t$-design of any fixed index in $\operatorname{Bil}_{q}(m,m')$ provided that $m'$ is sufficiently large.

\section{Saturation theorems in the Johnson scheme and the Grassmann scheme}\label{sect_johnson.and.grassmann}
For the Johnson scheme and its $q$-analog, which is known as the Grassmann scheme, the action of the incidence operators can be described using the \textit{symmetric Jordan chains} for the finite graded poset structures. The concept of symmetric Jordan chains can be found in \cite{srinivasan2011symmetric,srinivasan2014goldman,terwilliger1990incidence,bruijn1951set} and we follow the terminology of Srinivasan.\par 
Let $P$ be a finite graded poset with partial order $\leq$ and \textit{rank function} $\operatorname{rk}:P\to\mathbb{N}$. Let $V(P)$ denote the complex vector space spanned freely by the elements of the poset $P$, i.e. each element $v\in V(P)$ is a formal complex linear combination
\begin{align}\label{eq_formal.linear.combination.vector.in.V(P)}
	v=\sum_{x\in P}\alpha_{x}x,\quad\alpha_{x}\in\mathbb{C}.
\end{align}
The \textit{length} of $v\in V(P)$, written $\|v\|$, is defined by $\sqrt{\langle v,v\rangle}$ where $\langle\cdot,\cdot\rangle$ represents the standard inner product $\langle x,y\rangle=\delta_{x,y}$ on $V(P)$ with the Kronecker delta $\delta_{x,y}$. We say that a sequence $s=(v_{1},v_{2},\ldots,v_{h})$ \textit{starts at rank $\operatorname{rk}(v_{1})$} and \textit{ends at rank $\operatorname{rk}(v_{h})$}. Define the \textit{up operator} $\mathsf{U}$ on the poset $P$ by setting
\begin{align*}
	\mathsf{U}(x)=\sum_{y\in P,x\leq y\atop\operatorname{rk}(y)=\operatorname{rk}(x)+1}y
\end{align*}
and denote by $\mathsf{D}$ its adjoint operator, which is called the \textit{down operator}. Equivalently, $\mathsf{D}(y)=\sum_{x\in P,x\leq y,\operatorname{rk}(x)=\operatorname{rk}(y)-1}x$.
\begin{definition}
	With notation as above, a symmetric Jordan chain in $V(P)$ is a sequence
	\begin{align*}
		s=\left(v_{1},v_{2},\ldots,v_{h}\right),\quad\text{$v_{i}\in V(P)\backslash\left\{0\right\}$ for $i=1,2,\ldots,h$}
	\end{align*}
	satisfying the following conditions:
	\begin{enumerate}[itemsep=5pt]
		\item Each $v_{i}$ is homogeneous; that is, for each $i\in\left\{1,2,\ldots,h\right\}$, all $x\in P$ in the expansion \eqref{eq_formal.linear.combination.vector.in.V(P)} of $v_{i}$ with coefficient $\alpha_{x}\neq 0$ have the same rank. We denote this common rank by $\operatorname{rk}(v_{i})$ and call it the rank of $v_{i}$;
		\item $\mathsf{U}(v_{i-1})=v_{i}$ for $i=2,3,\ldots,h$ and $\mathsf{U}(v_{h})=0$;
		\item $\operatorname{rk}(v_{1})+\operatorname{rk}(v_{h})=\operatorname{rk}(P)$ if $h\geq 2$, or $2\operatorname{rk}(v_{1})=\operatorname{rk}(P)$ if $h=1$.
	\end{enumerate}
	An \textit{orthogonal symmetric Jordan basis} of $V(P)$ is an orthogonal basis of $V(P)$ consisting of a disjoint union of symmetric Jordan chains in $V(P)$.
\end{definition}
From this definition, the following fact is straightforward.
\begin{lemma}\label{lemma_graded.poset.down.operator.on.orthogonal.symmetric.Jordan.basis}
	With notation as above, for each symmetric Jordan chain $(v_{1},v_{2},\ldots,v_{h})$ in an orthogonal symmetric Jordan basis of $V(P)$,
	\begin{align*}
		\mathsf{D}v_{m+1}=\frac{\|v_{m+1}\|^{2}}{\|v_{m}\|^{2}}v_{m}
	\end{align*}
	holds for all $1\leq m\leq h-1$.
	\begin{proof}
		For every basis vector $w$ of the same rank as $v_{m}$,
		\begin{equation*}
			\langle\mathsf{D}v_{m+1},w\rangle=\langle v_{m+1},\mathsf{U}w\rangle.
		\end{equation*}
		Orthogonality of the chain basis makes this zero unless $w=v_{m}$. Thus $\mathsf{D}v_{m+1}=c_{m}v_{m}$. Taking inner products with $v_{m}$ gives $c_{m}\|v_{m}\|^2=\langle v_{m+1},\mathsf{U}v_{m}\rangle=\|v_{m+1}\|^{2}$.
	\end{proof}
\end{lemma}
\subsection{The Johnson scheme}
Let $v$ and $n$ be positive integers satisfying $2n\leq v$. Let $[v]=\left\{1,2,\ldots,v\right\}$ and let $X_{i}={[v]\choose i}$ be the collection of all $i$-subsets of $[v]$, where $0\leq i\leq n$. We denote $X:=X_{n}$. Then the symmetric group $\mathfrak{S}_{v}$ acts transitively on each $X_{i}$ and preserves inclusions. Let $\mathcal{P}_{\leq n}(v):=\left\{S\subseteq[v]\mid |S|\leq n\right\}$. Then $\mathcal{P}_{\leq n}(v)$ is a $\mathfrak{S}_{v}$-admissible graded poset with the partial order $\subseteq$ and the rank function $\operatorname{rk}(S):=\# S$. The top fiber of $\mathcal{P}_{\leq n}(v)$ is $X$. The following well-known result is called Young's rule, which can be found in \cite[\S 7.3, Corollary 1]{fulton1997young} and \cite[Theorem 29.13]{gordon2001representations}.
\begin{lemma}[Young's rule]\label{lemma_johnson.multiplicity.free.decomposition}
	With notation as above, we have an orthogonal multiplicity-free decomposition
	\begin{equation*}
		\mathbb{C}^{X}\cong U_{0}\oplus U_{1}\oplus\cdots\oplus U_{n},
	\end{equation*}
	where $U_{j}=S^{(v-j,j)}$ is the Specht module of shape $(v-j,j)$ and $\dim{U_{j}}={v\choose j}-{v\choose j-1}$ for each $j\in\left\{0,1,\ldots,n\right\}$.
\end{lemma}
The following result is summarized from Theorem 1.2 and the proof of Theorem 1.3 in \cite{srinivasan2011symmetric}, where every irreducible component $U_{j}$ is characterized using the orthogonal symmetric Jordan basis of $V\left(\mathcal{P}(v)\right)$ produced by the so-called \textit{de Bruijn--Tengbergen--Kruyswijk algorithm}.
\begin{lemma}[{\cite{srinivasan2011symmetric}}]\label{lemma_spanning.set.of.johnson.component.from.symmetric.Jordan.basis}
	There exists an orthogonal symmetric Jordan basis $J(v)$ of $V\left(\mathcal{P}(v)\right)$, where $\mathcal{P}(v)$ denotes the power set of $[v]$, such that the set
	\begin{align*}
		\left\{x\in J(v)\mid\text{$\operatorname{rk}(x)=n$ and the symmetric Jordan chain containing $x$ starts at rank $j$}\right\}
	\end{align*}
	spans the space $U_{j}$ defined in Lemma \ref{lemma_johnson.multiplicity.free.decomposition}. Furthermore, for any symmetric Jordan chain $(x_{j},\ldots,x_{v-j})$ in $J(v)$ starting at rank $j$ and ending at rank $v-j$, we have
	\begin{align*}
		\frac{\|x_{m+1}\|}{\|x_{m}\|}=\sqrt{(m+1-j)(v-j-m)}
	\end{align*}
	for all $j\leq m<v-j$.
\end{lemma}
\begin{lemma}\label{lemma_johnson.spectra.of.gram.matrix.of.incidence.operator}
	With notation as above, for $0\leq i\leq n$,
	\begin{align}\label{eq_johnson.decomposition.of.the.image.of.adjoint.inclusion.operator}
		\operatorname{Im}\left(W_{i}^*\right)=\operatorname{span}_{\mathbb{C}}\left\{\boldsymbol{1}_{\left\{T\in X\mid S\subseteq T\right\}}\mid S\in X_{i}\right\}=U_{0}\oplus U_{1}\oplus\cdots\oplus U_{i}.
	\end{align}
	Furthermore, for $0\leq j\leq i\leq n$, the operator $W_{i}^*W_{i}$ acts on $U_{j}$ as multiplication by
	\begin{align*}
		\gamma_{i,j}(v)={n-j\choose i-j}{v-i-j\choose n-i}
	\end{align*}
	whereas for $j>i$ the operator $W_{i}$ vanishes on $U_{j}$.
	\begin{proof}
		By Lemma \ref{lemma_general.spectra.of.gram.matrix.of.inclusion.operator}, it suffices to derive the values of all $\gamma_{i,j}$ since \eqref{eq_johnson.decomposition.of.the.image.of.adjoint.inclusion.operator} then holds automatically. For each $x\in{[v]\choose i}$, its image $\mathsf{U}^{n-i}(x)$ is the formal sum over all $x_{n}\in{[v]\choose n}$ given by the complete chains
		\begin{align*}
			x:=x_{i}\subsetneqq x_{i+1}\subsetneqq\cdots\subsetneqq x_{n},\quad\text{$\# x_{r}=r$ for all $i\leq r\leq n$}.
		\end{align*}
		Hence the coefficient of a fixed $y\in{[v]\choose n}$ in $\mathsf{U}^{n-i}(x)$ equals $(n-i)!$, which is the number of all complete chains from $x$ to $y$. This yields
		\begin{align*}
			\mathsf{U}^{n-i}(x)=(n-i)!\sum_{y\in{[v]\choose n},x\subseteq y}y=(n-i)!W_{i}^*(x).
		\end{align*}
		Taking adjoints of both sides gives $(n-i)!W_{i}=\mathsf{D}^{n-i}$.\par 
		Let $(x_{j},x_{j+1},\ldots,x_{v-j})$, where $j\leq i\leq n$, be a symmetric Jordan chain as described in Lemma \ref{lemma_spanning.set.of.johnson.component.from.symmetric.Jordan.basis}. Then from Lemma \ref{lemma_graded.poset.down.operator.on.orthogonal.symmetric.Jordan.basis} we have
		\begin{align}\label{eq_johnson.down.operator.on.orthogonal.symmetric.Jordan.basis}
			\mathsf{D}x_{m+1}=\frac{\|x_{m+1}\|^{2}}{\|x_{m}\|^{2}}x_{m}=(m+1-j)(v-j-m)x_{m}.
		\end{align}
		This yields
		\begin{align*}
			W_{i}^*W_{i}x_{n}=&W_{i}^*\left(\frac{1}{(n-i)!}\mathsf{D}^{n-i}x_{n}\right)\\
			=&\frac{1}{(n-i)!}\left(\prod_{m=i}^{n-1}(m+1-j)(v-j-m)\right)W_{i}^*x_{i}\\
			=&\frac{1}{\left((n-i)!\right)^{2}}\left(\prod_{m=i}^{n-1}(m+1-j)(v-j-m)\right)x_{n}\\
			=&\frac{(n-j)!}{(n-i)!(i-j)!}\frac{(v-i-j)!}{(n-i)!(v-n-j)!}x_{n}\\
			=&{n-j\choose i-j}{v-i-j\choose n-i}x_{n}.
		\end{align*}
		If $j>i$, then \eqref{eq_johnson.down.operator.on.orthogonal.symmetric.Jordan.basis} implies $W_{i}x_{n}=\mathsf{D}^{n-i}x_{n}=0$. This completes the proof.
	\end{proof}
\end{lemma}
\begin{lemma}\label{lemma_johnson.estimates.on.A_{i}(n).and.B_{i,j}(n)}
	We write
	\begin{align*}
		A_{i}(v):=\frac{b_{D}\gamma_{i,0}(v)}{|X|\alpha_{i}}\quad\text{and}\quad B_{i,j}(v):=\lambda\frac{\alpha_{j}\gamma_{i,j}(v)}{\alpha_{i}\gamma_{j,j}(v)}.
	\end{align*}
	Then there exists $v_{0}:=v_{0}(n,t,\lambda)$ such that for every $i\in\left\{t+1,t+2,\ldots,n\right\}$ the inequality $A_{i}(v)+\sum_{j=t+1}^{i-1}B_{i,j}(v)<1$ holds for all $v\geq v_{0}$.
	\begin{proof}
		Note that
		\begin{equation*}
			A_{i}(v)=\frac{\lambda{v\choose t}{n\choose i}{v-i\choose n-i}}{{n\choose t}{v\choose n}{n\choose i}}=\lambda\frac{{v-i\choose n-i}}{{v-t\choose n-t}}=\lambda\left(1+o(1)\right)\frac{(n-t)!(v-i)^{n-i}}{(n-i)!(v-t)^{n-t}}=O\left(v^{-(i-t)}\right)
		\end{equation*}
		and
		\begin{equation*}
			B_{i,j}(v)=\lambda\frac{{n\choose j}{n-j\choose i-j}{v-i-j\choose n-i}}{{n\choose i}{v-2j\choose n-j}}=\lambda{i\choose j}\frac{{v-i-j\choose n-i}}{{v-2j\choose n-j}}=O\left(v^{-(i-j)}\right).\qedhere
		\end{equation*}
	\end{proof}
\end{lemma}
Lemmas \ref{lemma_johnson.multiplicity.free.decomposition}, \ref{lemma_johnson.spectra.of.gram.matrix.of.incidence.operator}, and \ref{lemma_johnson.estimates.on.A_{i}(n).and.B_{i,j}(n)} verify all the hypotheses of Theorem \ref{thm_general.saturation.theorem} for the Johnson scheme when $v$ is sufficiently large.
\begin{theorem}\label{thm_johnson.asymptotic.saturation}
	Fix $n,t,\lambda$. There exists a constant $v_{0}=v_{0}(n,t,\lambda)$ such that, for every $v\geq v_{0}$ and every combinatorial $t$-$(v,n,\lambda)$ design $D$, it holds that
	\begin{align*}
		\operatorname{span}_{\mathbb{C}}\left\{g\boldsymbol{1}_{D}\mid g\in\mathfrak{S}_{v}\right\}=U_{0}\oplus U_{t+1}\oplus U_{t+2}\oplus\cdots\oplus U_{n},
	\end{align*}
	where each $U_{j}$ is as described in Lemma \ref{lemma_johnson.multiplicity.free.decomposition}. In particular,
	\begin{align*}
		\dim_{\mathbb{Q}}\operatorname{span}_{\mathbb{Q}}\left\{g\boldsymbol{1}_{D}\mid g\in\mathfrak{S}_{v}\right\}={v\choose n}-{v\choose t}+1.
	\end{align*}
\end{theorem}
We recall that, by Keevash's existence theorem \cite{keevash2014existence}, for fixed $n,t,\lambda$ the divisibility conditions are sufficient for all sufficiently large admissible values of $v$. Consequently, Theorem \ref{thm_johnson.asymptotic.saturation} is non-vacuous for all such parameters. Meanwhile, its dimension conclusion recovers \cite[Theorem 4]{ghorbani2025vector}, which additionally provides an explicit lower bound on $v$.
\subsection{The Grassmann scheme}
Let $q>1$ be a prime power and let $\mathbb{F}_{q}$ be the finite field with $q$ elements. The set of all $k$-dimensional subspaces of $\mathbb{F}_{q}^{n}$ is called the \emph{Grassmannian} over $\mathbb{F}_{q}$ and is denoted by $\operatorname{Gr}_{q}(n,k)$. It is well known that $\#\operatorname{Gr}_{q}(n,k)={n\brack k}_{q}$ for all $0\leq k\leq n$. Let ${n\brack 1}_{q}$ be denoted by the \textit{$q$-analog integer} $[n]_{q}$.\par 
Let $\mathcal{L}_{q}(k)$ be the graded poset of all subspaces of $\mathbb{F}_{q}^{n}$ with dimension less than or equal to $k$, which is ordered by inclusion and the rank of each subspace $L\in\mathcal{L}_{q}(k)$ is $\operatorname{rk}(L):=\dim{L}$. Then $\mathcal{L}_{q}(k)$ is a $\operatorname{GL}(n,q)$-admissible graded poset, with $X:=\operatorname{Gr}_{q}(n,k)$ being its top fiber.\par 
Assume that $n\geq 2k$. A $t$-$(n,k,\lambda)_{q}$ design in the Grassmann scheme, also known as a \textit{$t$-$(n,k,\lambda)_{q}$ subspace design}, is a subset $D$ of the Grassmannian $\operatorname{Gr}_{q}(n,k)$ such that every $t$-dimensional subspace $T$ of $\mathbb{F}_{q}^{n}$ is contained in exactly $\lambda$ blocks of $D$. The following $q$-analog result of Lemma \ref{lemma_johnson.multiplicity.free.decomposition} follows from \cite[Theorem 3.1]{srinivasan2014goldman}.
\begin{lemma}[{\cite{srinivasan2014goldman}}]\label{lemma_grassmann.multiplicity.free.decomposition}
	There exists an orthogonal multiplicity-free decomposition
	\begin{align*}
		\mathbb{C}^{X}=U_{0}\oplus U_{1}\oplus\cdots\oplus U_{k},
	\end{align*}
	where each $U_{i}$ is an irreducible $\operatorname{GL}(n,q)$-module and has dimension ${n\brack i}_{q}-{n\brack i-1}_{q}$.
\end{lemma}
The following result is a consequence of \cite[Theorem 1.3]{srinivasan2014goldman} and the proof of \cite[Theorem 3.2]{srinivasan2014goldman}. The formula for the quotient $\|v_{m+1}\|/\|v_{m}\|$ was proved earlier by Terwilliger \cite[Item 5 of Theorem 3.3]{terwilliger1990incidence}.
\begin{lemma}[\cite{srinivasan2014goldman,terwilliger1990incidence}]\label{lemma_spanning.set.of.grassmann.component.from.symmetric.Jordan.basis}
	There exists an orthogonal symmetric Jordan basis $J_{q}(n)$ of $V(\mathcal{L}_{q}(n))$ such that the set
	\begin{align*}
		\left\{v\in J_{q}(n)\mid\text{$\operatorname{rk}(v)=k$ and the symmetric Jordan chain containing $v$ starts at rank $j$}\right\}
	\end{align*}
	spans the space $U_{j}$ defined in Lemma \ref{lemma_grassmann.multiplicity.free.decomposition}. Furthermore, for any symmetric Jordan chain $(v_{j},\ldots,v_{n-j})$ in $J_{q}(n)$ starting at rank $j$ and ending at rank $n-j$, we have
	\begin{align*}
		\frac{\|v_{m+1}\|}{\|v_{m}\|}=\sqrt{q^{j}[m+1-j]_{q}[n-j-m]_{q}}
	\end{align*}
	for all $j\leq m<n-j$.
\end{lemma}
\begin{lemma}\label{lemma_grassmann.spectra.of.gram.matrix.of.incidence.operator}
	With notation as above, for $0\leq i\leq k$,
	\begin{align}\label{eq_grassmann.decomposition.of.the.image.of.adjoint.inclusion.operator}
		\operatorname{Im}\left(W_{i}^*\right)=\operatorname{span}_{\mathbb{C}}\left\{\boldsymbol{1}_{\left\{M\in\operatorname{Gr}_{q}(n,k)\mid L\subseteq M\right\}}\mid L\in\operatorname{Gr}_{q}(n,i)\right\}=U_{0}\oplus U_{1}\oplus\cdots\oplus U_{i}.
	\end{align}
	Furthermore, for $0\leq j\leq i\leq k$, the operator $W_{i}^*W_{i}$ acts on $U_{j}$ as multiplication by
	\begin{align*}
		\gamma_{i,j}(n)=q^{j(k-i)}{k-j\brack i-j}_{q}{n-i-j\brack k-i}_{q}
	\end{align*}
	whereas for $j>i$ the operator $W_{i}$ vanishes on $U_{j}$.
	\begin{proof}
		By the definition, for each $L\in\operatorname{Gr}_{q}(n,i)$ the image $\mathsf{U}^{k-i}(L)$ is the formal sum over all $L_{k}$ given by the complete flags
		\begin{align*}
			L=L_{i}\subsetneqq L_{i+1}\subsetneqq\cdots\subsetneqq L_{k},\quad\text{$\dim{L_{r}}=r$ for all $i\leq r\leq k$}.
		\end{align*}
		Counting the one-dimensional subspaces of $M/L$ we find that the coefficient of a fixed $M\in\operatorname{Gr}_{q}(n,k)$ in $\mathsf{U}^{k-i}(L)$ equals the $q$-factorial $[k-i]_{q}!:=\prod_{r=1}^{k-i}[r]_{q}$. This yields
		\begin{align*}
			\mathsf{U}^{k-i}(L)=[k-i]_{q}!\sum_{M\in\operatorname{Gr}_{q}(n,k),L\subseteq M}M=[k-i]_{q}!W_{i}^*(L).
		\end{align*}
		Let $(v_{j},\ldots,v_{n-j})$ be any symmetric Jordan chain in $J_{q}(n)$ starting at rank $j$ and ending at rank $n-j$, where $j\leq m<n-j$. Then by Lemma \ref{lemma_spanning.set.of.grassmann.component.from.symmetric.Jordan.basis}, we find that
		\begin{align}\label{eq_grassmann.adjoint.up.operator.shift}
			\mathsf{D}v_{m+1}=\frac{\|v_{m+1}\|^{2}}{\|v_{m}\|^{2}}v_{m}=q^{j}[m+1-j]_{q}[n-j-m]_{q}v_{m}.
		\end{align}
		If $j\leq i$, this yields
		\begin{align*}
			W_{i}^*W_{i}v_{k}=\frac{1}{\left([k-i]_{q}!\right)^{2}}\left(\prod_{m=i}^{k-1}q^{j}[m+1-j]_{q}[n-j-m]_{q}\right)v_{k}=q^{j(k-i)}{k-j\brack i-j}_{q}{n-i-j\brack k-i}_{q}v_{k}.
		\end{align*}
		For each $j>i$, from \eqref{eq_grassmann.adjoint.up.operator.shift} we see that the operator $W_{i}$ annihilates $U_{j}$. This completes the proof.
	\end{proof}
\end{lemma}
\begin{lemma}\label{lemma_grassmann.estimates.on.A_{i}(n).and.B_{i,j}(n)}
	With notation as above, we set
	\begin{align*}
		A_{i}(n):=\frac{b_{D}\gamma_{i,0}(n)}{{k\brack i}_{q}{n\brack k}_{q}}\quad\text{and}\quad B_{i,j}(n):=\lambda\frac{\gamma_{i,j}(n){k\brack j}_{q}}{\gamma_{j,j}(n){k\brack i}_{q}}.
	\end{align*}
	Then there exists $n_{0}:=n_{0}(q,k,t,\lambda)$ such that for every $i\in\left\{t+1,t+2,\ldots,k\right\}$,
	\begin{align*}
		A_{i}(n)+\sum_{j=t+1}^{i-1}B_{i,j}(n)<1
	\end{align*}
	holds for all $n\geq n_{0}$.
	\begin{proof}
		We compute
		\begin{align*}
			A_{i}(n)=\frac{b_{D}\gamma_{i,0}(n)}{{k\brack i}_{q}{n\brack k}_{q}}=\lambda\frac{{n\brack t}_{q}{n-i\brack k-i}_{q}}{{k\brack t}_{q}{n\brack k}_{q}}=\lambda\frac{{n-i\brack k-i}_{q}}{{n-t\brack k-t}_{q}}=\Theta\left(q^{-(i-t)(n-k)}\right)
		\end{align*}
		and
		\begin{multline*}
			B_{i,j}(n)=\lambda\frac{\gamma_{i,j}(n){k\brack j}_{q}}{\gamma_{j,j}(n){k\brack i}_{q}}=\lambda q^{-j(i-j)}{k-j\brack i-j}_{q}\frac{{n-i-j\brack k-i}_{q}{k\brack j}_{q}}{{n-2j\brack k-j}_{q}{k\brack i}_{q}}\\=\Theta\left(q^{-j(i-j)+(i-j)(k-i)-(i-j)(n-k-j)-(i-j)(k-i-j)}\right)=\Theta\left(q^{-(n-k-j)(i-j)}\right).
		\end{multline*}
		Therefore $A_{i}(n)+\sum_{j=t+1}^{i-1}B_{i,j}(n)=O\left(q^{-(n-k)}\right)=o_{n}(1)$, which gives the result.
	\end{proof}
\end{lemma}
Hence, Lemmas \ref{lemma_grassmann.multiplicity.free.decomposition}, \ref{lemma_grassmann.spectra.of.gram.matrix.of.incidence.operator}, and \ref{lemma_grassmann.estimates.on.A_{i}(n).and.B_{i,j}(n)} verify all the hypotheses of Theorem \ref{thm_general.saturation.theorem} for the Grassmann scheme when the dimension $n$ of the ambient space is sufficiently large.
\begin{theorem}\label{thm_grassmann.asymptotic.saturation}
	Fix $q,k,t,\lambda$. There exists a constant $n_{0}=n_{0}(q,k,t,\lambda)$ such that, for every $n\geq n_{0}$ and every $t$-$(n,k,\lambda)_{q}$ subspace design $D$, it holds that
	\begin{align*}
		\operatorname{span}_{\mathbb{C}}\left\{g\boldsymbol{1}_{D}\mid g\in\operatorname{GL}(n,q)\right\}=U_{0}\oplus U_{t+1}\oplus U_{t+2}\oplus\cdots\oplus U_{k},
	\end{align*}
	where each $U_{j}$ is as described in Lemma \ref{lemma_grassmann.multiplicity.free.decomposition}. In particular,
	\begin{align*}
		\dim_{\mathbb{Q}}\operatorname{span}_{\mathbb{Q}}\left\{g\boldsymbol{1}_{D}\mid g\in\operatorname{GL}(n,q)\right\}={n\brack k}_{q}-{n\brack t}_{q}+1.
	\end{align*}
\end{theorem}
By the asymptotic existence theorem of Keevash, Sah and Sawhney \cite[Theorem 1.4]{keevash2025existence}, for fixed $q,k,t,\lambda$ the divisibility conditions ${k-i\brack t-i}_{q}\mid\lambda{n-i\brack t-i}_{q}$, $0\leq i\leq t$ are sufficient for all sufficiently large admissible values of $n$. Thus the conclusion of Theorem \ref{thm_grassmann.asymptotic.saturation} is realized along all sufficiently large admissible dimensions.\par 
In the asymptotic regime above, the dimension conclusion presented strengthens \cite[Theorem 1.6]{li2026dimension.q-steiner}, which concerns the spanning space of all characteristic vectors of $t$-$(n,k,1)_{q}$ subspace designs. Moreover, Theorem \ref{thm_grassmann.asymptotic.saturation} shows that the orbit of each individual design already spans the same space.

\section*{Declaration of AI use}
During the preparation of this paper, the authors used generative artificial intelligence (AI) tools to assist with proofreading and improving the language of the manuscript. AI tools were also used for discussion of possible proof strategies in the Johnson and Grassmann scheme cases; in particular, they suggested the use of orthogonal symmetric Jordan bases for studying the relevant incidence operators. All mathematical arguments were developed, checked, and verified by the authors, who take full responsibility for the content of this paper.

\section*{Acknowledgment}
This work is supported by the National Natural Science Foundation of China (No.\ 12371337) and the Natural Science Foundation of Hunan Province (No.\ 2023RC1003).


\bibliographystyle{amsalpha}
\bibliography{references}

@article{ghorbani2025vector,
	author = {Ghorbani, E. and Kamali, S. and Khosrovshahi, G. B.},
	title = {{The vector space generated by permutations of a trade or a design}},
	fjournal = {Journal of Combinatorial Theory. Series A},
	journal = {J. Comb. Theory, Ser. A},
	issn = {0097-3165},
	volume = {210},
	pages = {15},
	note = {Id/No 105969},
	year = {2025},
	doi = {10.1016/j.jcta.2024.105969},
	zbMATH = {7955731},
	Zbl = {1553.05023}
}

@article{srinivasan2014goldman,
	author = {Srinivasan, Murali K.},
	title = {{The {Goldman}--{Rota} identity and the {Grassmann} scheme}},
	fjournal = {The Electronic Journal of Combinatorics},
	journal = {Electron. J. Comb.},
	issn = {1077-8926},
	volume = {21},
	number = {1},
	pages = {23},
	note = {Id/No p1.37},
	year = {2014},
	url = {www.combinatorics.org/ojs/index.php/eljc/article/view/v21i1p37},
	zbMATH = {6340135},
	Zbl = {1300.05324}
}

@book{gordon2001representations,
	author = {James, Gordon and Liebeck, Martin},
	title = {{Representations and Characters of Groups}},
	edition = {2nd},
	isbn = {0-521-00392-X; 0-521-81205-4},
	year = {2001},
	publisher = {Cambridge: Cambridge University Press},
	zbMATH = {1666801},
	Zbl = {0981.20004}
}

@incollection{ceccherini2009finite,
	author = {Ceccherini-Silberstein, Tullio and d'Angeli, Daniele and Donno, Alfredo and Scarabotti, Fabio and Tolli, Filippo},
	title = {{Finite {Gelfand} pairs: examples and applications}},
	booktitle = {Ischia group theory 2008. Proceedings of the conference in group theory, Naples, Italy, April 1--4, 2008.},
	isbn = {978-981-4277-79-2},
	pages = {7--41},
	year = {2009},
	publisher = {Hackensack, NJ: World Scientific},
	zbMATH = {5657531},
	Zbl = {1184.43010}
}

@book{macdonald1995symmetric,
	author = {Macdonald, Ian Grant},
	title = {{Symmetric Functions and {Hall} Polynomials}},
	edition = {2nd},
	isbn = {0-19-853489-2},
	year = {1995},
	publisher = {Oxford: Clarendon Press},
	zbMATH = {739282},
	Zbl = {0824.05059}
}

@article{martin2009commutative,
	author = {Martin, William J. and Tanaka, Hajime},
	title = {{Commutative association schemes}},
	fjournal = {European Journal of Combinatorics},
	journal = {Eur. J. Comb.},
	issn = {0195-6698},
	volume = {30},
	number = {6},
	pages = {1497--1525},
	year = {2009},
	doi = {10.1016/j.ejc.2008.11.001},
	zbMATH = {5640333},
	Zbl = {1228.05317}
}

@book{fulton1997young,
	author = {Fulton, William},
	title = {{Young Tableaux: {With} Applications to Representation Theory and Geometry}},
	fseries = {London Mathematical Society Student Texts},
	series = {Lond. Math. Soc. Stud. Texts},
	issn = {0963-1631},
	volume = {35},
	isbn = {0-521-56724-6},
	year = {1997},
	publisher = {Cambridge: Cambridge University Press},
	zbMATH = {1001729},
	Zbl = {0878.14034}
}

@article{srinivasan2011symmetric,
	author = {Srinivasan, Murali K.},
	title = {{Symmetric chains, {Gelfand}--{Tsetlin} chains, and the {Terwilliger} algebra of the binary {Hamming} scheme}},
	fjournal = {Journal of Algebraic Combinatorics},
	journal = {J. Algebr. Comb.},
	issn = {0925-9899},
	volume = {34},
	number = {2},
	pages = {301--322},
	year = {2011},
	doi = {10.1007/s10801-010-0272-2},
	zbMATH = {5968629},
	Zbl = {1229.05298}
}

@book{brouwer1989distance,
	author = {Brouwer, Andries E. and Cohen, Arjeh M. and Neumaier, Arnold},
	title = {{Distance-regular Graphs}},
	fseries = {Ergebnisse der Mathematik und ihrer Grenzgebiete. 3. Folge},
	series = {Ergeb. Math. Grenzgeb., 3. Folge},
	issn = {0071-1136},
	volume = {18},
	isbn = {3-540-50619-5},
	year = {1989},
	publisher = {Berlin etc.: Springer-Verlag},
	zbMATH = {43547},
	Zbl = {0747.05073}
}

@article{delsarte1978bilinear,
	author = {Delsarte, Philippe},
	title = {{Bilinear forms over a finite field, with applications to coding theory}},
	fjournal = {Journal of Combinatorial Theory. Series A},
	journal = {J. Comb. Theory, Ser. A},
	issn = {0097-3165},
	volume = {25},
	pages = {226--241},
	year = {1978},
	doi = {10.1016/0097-3165(78)90015-8},
	zbMATH = {3618036},
	Zbl = {0397.94012}
}

@book{bannai2021algebraic,
	author = {Bannai, Eiichi and Bannai, Etsuko and Ito, Tatsuro and Tanaka, Rie},
	title = {{Algebraic Combinatorics}},
	edition = {Originally published by {Kyoritsu} {Shuppan} ({Kyoritsu} {Publisher}), {Tokyo} 2016},
	fseries = {De Gruyter Series in Discrete Mathematics and Applications},
	series = {De Gruyter Ser. Discrete Math. Appl.},
	issn = {2195-5557},
	volume = {5},
	isbn = {978-3-11-062763-3; 978-3-11-063025-1},
	year = {2021},
	publisher = {Berlin: De Gruyter},
	doi = {10.1515/9783110630251},
	zbMATH = {7279831},
	Zbl = {1472.05002}
}

@article{li2026dimension.mrd,
	author = {Li, Qilong and Zhou, Yue},
	title = {{On the dimension of the space generated by characteristic vectors of MRD codes}},
	journal = {Journal of Combinatorial Designs},
	volume = {34},
	number = {8},
	pages = {345--360},
	doi = {https://doi.org/10.1002/jcd.70018},
	url = {https://onlinelibrary.wiley.com/doi/abs/10.1002/jcd.70018},
	eprint = {https://onlinelibrary.wiley.com/doi/pdf/10.1002/jcd.70018},
	year = {2026}
}

@article{bruijn1951set,
	author = {de Bruijn, N. G. and van Ebbenhorst Tengbergen, C. and Kruyswijk, D.},
	title = {{On the set of divisors of a number}},
	fjournal = {Nieuw Archief voor Wiskunde. Tweede Serie},
	journal = {Nieuw Arch. Wiskd., II. Ser.},
	issn = {0028-9825},
	volume = {23},
	pages = {191--193},
	year = {1951},
	zbMATH = {3065933},
	Zbl = {0043.04301}
}

@misc{terwilliger1990incidence,
	author = {Terwilliger, Paul},
	title = {{The incidence algebra of a uniform poset}},
	year = {1990},
	howpublished = {{Coding Theory and Design Theory}. {Part} {I}: {Coding} {Theory}, {Proc}. {Workshop} {IMA} {Program} {Appl}. {Comb}., {Minneapolis}/{MN} ({USA}) 1987-88, {IMA} {Vol}. {Math}. {Appl}. 21, 193--212 (1990).},
	zbMATH = {15444},
	Zbl = {0737.05032}
}

@book{stanley2012enumerative,
	author = {Stanley, Richard P.},
	title = {{Enumerative Combinatorics. {Vol}. 1}},
	edition = {2nd},
	fseries = {Cambridge Studies in Advanced Mathematics},
	series = {Camb. Stud. Adv. Math.},
	volume = {49},
	isbn = {978-1-107-60262-5; 978-1-107-01542-5; 978-1-139-20056-1},
	year = {2012},
	publisher = {Cambridge: Cambridge University Press},
	url = {www.cambridge.org/de/knowledge/isbn/item6832283/?site_locale=de_DE},
	zbMATH = {6016068},
	Zbl = {1247.05003}
}

@article{kung1993radon,
	title = {{The Radon transforms of a combinatorial geometry. 2. Partition lattices}},
	journal = {Advances in Mathematics},
	volume = {101},
	number = {1},
	pages = {114-132},
	year = {1993},
	issn = {0001-8708},
	doi = {https://doi.org/10.1006/aima.1993.1044},
	url = {https://www.sciencedirect.com/science/article/pii/S0001870883710443},
	author = {Joseph P.S. Kung}
}

@book{finch2003mathematical,
	author = {Finch, Steven R.},
	title = {{Mathematical Constants}},
	fseries = {Encyclopedia of Mathematics and Its Applications},
	series = {Encycl. Math. Appl.},
	issn = {0953-4806},
	volume = {94},
	isbn = {0-521-81805-2},
	year = {2003},
	publisher = {Cambridge: Cambridge University Press},
	zbMATH = {2018401},
	Zbl = {1054.00001}
}

@article{delsarte1976association,
	author = {Delsarte, Philippe},
	title = {{Association schemes and $t$-designs in regular semilattices}},
	fjournal = {Journal of Combinatorial Theory. Series A},
	journal = {J. Comb. Theory, Ser. A},
	issn = {0097-3165},
	volume = {20},
	pages = {230--243},
	year = {1976},
	doi = {10.1016/0097-3165(76)90017-0},
	zbMATH = {3532976},
	Zbl = {0342.05020}
}

@article{graver1973module,
	author = {Graver, J. E. and Jurkat, W. B.},
	title = {{The module structure of integral designs}},
	fjournal = {Journal of Combinatorial Theory. Series A},
	journal = {J. Comb. Theory, Ser. A},
	issn = {0097-3165},
	volume = {15},
	pages = {75--90},
	year = {1973},
	doi = {10.1016/0097-3165(73)90037-X},
	zbMATH = {3409370},
	Zbl = {0259.05020}
}

@article{wilson1973necessary,
	author = {Wilson, Richard M.},
	title = {{The necessary conditions for $t$-designs are sufficient for something}},
	fjournal = {Utilitas Mathematica},
	journal = {Util. Math.},
	issn = {0315-3681},
	volume = {4},
	pages = {207--215},
	year = {1973},
	zbMATH = {3448595},
	Zbl = {0286.05005}
}

@article{gottlieb1966certain,
	author = {Gottlieb, Daniel H.},
	title = {{A certain class of incidence matrices}},
	fjournal = {Proceedings of the American Mathematical Society},
	journal = {Proc. Am. Math. Soc.},
	issn = {0002-9939},
	volume = {17},
	pages = {1233--1237},
	year = {1966},
	doi = {10.2307/2035716},
	zbMATH = {3235243},
	Zbl = {0146.01302}
}

@article{ghodrati2019dimension,
	author = {Ghodrati, Amir Hossein},
	title = {{Dimension of the space generated by {Steiner} systems}},
	fjournal = {Journal of Combinatorial Designs},
	journal = {J. Comb. Des.},
	issn = {1063-8539},
	volume = {27},
	number = {5},
	pages = {295--310},
	year = {2019},
	doi = {10.1002/jcd.21649},
	zbMATH = {7087251},
	Zbl = {1416.05047}
}

@article{keevash2025existence,
	author = {Keevash, Peter and Sah, Ashwin and Sawhney, Mehtaab},
	title = {{The existence of subspace designs}},
	fjournal = {Proceedings of the London Mathematical Society. Third Series},
	journal = {Proc. Lond. Math. Soc. (3)},
	issn = {0024-6115},
	volume = {131},
	number = {1},
	pages = {72},
	note = {Id/No e70071},
	year = {2025},
	doi = {10.1112/plms.70071},
	zbMATH = {8076212},
	Zbl = {1570.05016}
}

@article{keevash2014existence,
	title={The existence of designs},
	author={Keevash, Peter},
	journal={arXiv preprint arXiv:1401.3665},
	year={2014}
}

@article{li2026dimension.q-steiner,
	title={{On the dimension of the space generated by characteristic vectors of $q$-Steiner systems}},
	author={Li, Qilong and Wei{\ss}, Charlene and Zhou, Yue},
	journal={arXiv preprint arXiv:2605.06369},
	year={2026}
}

@article{maliakas2026total,
	author = {Maliakas, Mihalis and Stergiopoulou, Dimitra-Dionysia},
	title = {{Total trades, intersection matrices and {Specht} modules}},
	fjournal = {Linear Algebra and its Applications},
	journal = {Linear Algebra Appl.},
	issn = {0024-3795},
	volume = {732},
	pages = {74--92},
	year = {2026},
	doi = {10.1016/j.laa.2025.11.021},
	zbMATH = {8143577}
}

@book{hedayat1999orthogonal,
	author = {Hedayat, A. S. and Sloane, N. J. A. and Stufken, John},
	title = {{Orthogonal Arrays. {Theory} and Applications}},
	fseries = {Springer Series in Statistics},
	series = {Springer Ser. Stat.},
	issn = {0172-7397},
	isbn = {0-387-98766-5},
	year = {1999},
	publisher = {New York, NY: Springer},
	zbMATH = {1321696},
	Zbl = {0935.05001}
}

@book{roth2006introduction,
	author = {Roth, Ron M.},
	title = {{Introduction to Coding Theory}},
	isbn = {0-521-84504-1},
	year = {2006},
	publisher = {Cambridge: Cambridge University Press},
	doi = {10.1017/CBO9780511808968},
	zbMATH = {5016445},
	Zbl = {1092.94001}
}

@article{ray1988existence,
	author = {Ray-Chaudhuri, D. K. and Singhi, N. M.},
	title = {{On existence and number of orthogonal arrays}},
	fjournal = {Journal of Combinatorial Theory. Series A},
	journal = {J. Comb. Theory, Ser. A},
	issn = {0097-3165},
	volume = {47},
	number = {1},
	pages = {28--36},
	year = {1988},
	doi = {10.1016/0097-3165(88)90041-6},
	zbMATH = {4059419},
	Zbl = {0649.05017}
}

@article{kuperberg2017probabilistic,
	author = {Kuperberg, Greg and Lovett, Shachar and Peled, Ron},
	title = {{Probabilistic existence of regular combinatorial structures}},
	fjournal = {Geometric and Functional Analysis. GAFA},
	journal = {Geom. Funct. Anal.},
	issn = {1016-443X},
	volume = {27},
	number = {4},
	pages = {919--972},
	year = {2017},
	doi = {10.1007/s00039-017-0416-9},
	zbMATH = {6766923},
	Zbl = {1369.05024}
}

\end{document}